\documentclass[11pt]{amsart}
\usepackage{enumerate}
\usepackage{amsopn,amsfonts,amsmath,amssymb,amsthm}
\usepackage[dvips]{graphicx}
\usepackage{float}
\usepackage{subfigure}
\usepackage{dcolumn}
\usepackage{color, colortbl}
\usepackage[usenames,dvipsnames,svgnames,table]{xcolor}
\usepackage[colorlinks=true,
            linkcolor=red,	
            urlcolor=blue,
            citecolor=gray]{hyperref}
\allowdisplaybreaks

\usepackage{geometry}
\def\indic{{\rm {\large 1}\hspace{-2.3pt}{\large
l}}}

\def\R{{\mathbb R}}

\def\N{{\mathbb N}}

\def\C{{\mathbb C}}

\def\sech{{\rm sech}}

\newcommand{\df}[2]{\frac{#1 }{#2}}

\newcommand{\braket}[1]{\left\langle #1 \right\rangle }
\newcommand{\abs}[1]{ \left|  #1\right| }

\theoremstyle{definition}
\newtheorem{theorem}{Theorem}
\newtheorem{proposition}[theorem]{Proposition}
\newtheorem{lemma}[theorem]{Lemma}
\newtheorem{corollary}[theorem]{Corollary}

\usepackage[foot]{amsaddr}

\begin{document}
\title[]{
Estimates for the SVD of the truncated Fourier transform 
on $L^2(\cosh(b|\cdot|))$ 
and stable analytic continuation 
}


\author[Gaillac]{Christophe Gaillac$^{{(1),(2)}}$}
\address{$ ^{(1)}$ Toulouse School of Economics, Toulouse Capitole University, 21 all\'ee de Brienne, 31000 Toulouse, France.}
\address{$ ^{(2)}$ 
CREST, ENSAE, 5 avenue Henry Le Chatelier, 91764 Palaiseau, France.}
\email{\href{mailto:christophe.gaillac@tse-fr.eu}
{christophe.gaillac@tse-fr.eu}}

\author[Gautier]{Eric Gautier$^{(1)}$}
\email{\href{mailto:eric.gautier@tse-fr.eu}{eric.gautier@tse-fr.eu}}
\date{This version: \today. Preprint \href{https://arxiv.org/abs/1905.11338}{arXiv:1905.11338}} 

\thanks{Eric Gautier acknowledges financial support from the grant ERC POEMH 337665 and ANR-17-EURE-0010. The authors are very grateful to Aline Bonami for discussions on her work and sending us  \cite{bonami2016}}
\begin{abstract}
The Fourier transform truncated on $[-c,c]$ has for a long time been analyzed as acting on $L^2(-1/b,1/b)$ into $L^2(-1,1)$ and its right-singular vectors are the prolate spheroidal wave functions.
This paper  paper considers the operator defined on the larger space $L^2(\cosh(b|\cdot|))$ on which it remains injective. The main purpose is (1) to provide nonasymptotic upper and lower bounds on the singular values with similar qualitative behavior in $m$ (the index), $b$, and $c$ and (2) to derive nonasymptotic upper bounds on the sup-norm of the right-singular functions. Finally, we propose a numerical method to compute the SVD. These are fundamental results for \cite{estimation}. This paper, considers as an illustrative example 
analytic continuation of nonbandlimited functions 
when the function is observed with error on an interval.
\end{abstract} 

\maketitle

\section{Introduction}\label{sec:intro}	
Extrapolating an analytic square integrable function $ f $ from its observation with error on $[-c,c]$ to $\R$ 
has a wide range of applications, for example in imaging and signal processing \cite{Gosse}, in geostatistics and with big data \cite{coifman2006geometric}, and finance \cite{gosse2010analysis}. A researcher may want to estimate a density from censored data. This means that the data is only available on a smaller set than the one of interest (see, \emph{e.g.}, \cite{belitser1998efficient,berman2007legendre}). When the function is a Fourier transform, this is a type of super-resolution in image restoration \cite{Iv_imagaing,bertero1984resolution,gerchberg1974super} which can be achieved under auxiliary information such as information on the support of the object. A related problem of out-of-band extrapolation (see, \emph{e.g.}, \cite{Marechal1,Iv_imagaing}) consists in recovering a function from partial observation of its Fourier transform. This is different from extrapolation because the object of interest is indirectly observed and the reconstruction does not rely on 
extrapolation of the Fourier transform. 
A particular instance of this framework appears in the analysis of the random coefficients linear model (see, \emph{e.g.} \cite{estimation}). There, the model takes the form
\begin{equation}\label{eq}
Y = \alpha + \beta^{\top}X,
\end{equation} where $(\alpha,\beta^{\top})\in \R^{p+1}$
and $X\in \R^p$ are  independent random vectors, and the researcher has an independent and identically distributed sample of $(Y,X^{\top})$ from which he can 
estimate the Fourier transform of the density of the coefficients $(\alpha,\beta^{\top})$ on $\{ (t,tx): \ (t,x) \in \R\times \mathcal{X}\}$, where $\mathcal{X}\subseteq \R^p$ is the support of $X$ and the object of interest is the density the coefficients. In this problem, $\alpha$ is unbounded and if ever $\beta$ is bounded the bounds are usually unknown. 
This gives rise to an inverse problem generalizing inverse problems involving the Radon transform with partial observations. In the context of tomography, the objects are unbounded or with unknown bounds  (see \cite{estimation}) and the angles are randomly drawn, have a limited support, and the dimension is arbitrary. 

It is customary to rely on analytic functions and use Hilbert space techniques. 
For the extrapolation problem, one can restrict attention to bandlimited functions  which are square integrable functions whose Fourier transforms have support in 
$[-1/b,1/b]$. For out-of-band extrapolation, one can work with square-integrable functions whose support is a subset of $[-1/b,1/b]$ in which case their Fourier transform is analytic by the Paley-Wiener theorem (see \cite{ReedSimon}).  
Prolate spheroidal wave functions (henceforth PSWF, see \cite{Osipov,Slepian2}), also called ``Slepian functions", are the right-singular functions of the truncated Fourier transform restricted  to functions with support in $[-1/b,1/b]$. The truncated Fourier transform maps functions to their Fourier transform on $[-c,c]$. 
The PSWF form an orthonormal basis of the space $L^2(-c,c)$ of square-integrable functions on $(-c,c)$, are restrictions of square integrable orthogonal analytic functions on $\mathbb{R}$, and form a complete system of the bandlimited functions with bandlimits $[-1/b,1/b]$.  Hence, a bandlimited function on the whole line is simply the series expansion on the PSWF basis, also called Slepian series, whose coefficients only depend on the function on $(-c,c)$, almost everywhere on $\R$. This makes sense if we understand the PSWF functions as their extension to $\R$. In this framework, analytic continuation is an inverse problem in the sense that the solution does not depend continuously on the data, 
more specifically severely ill-posed (see, \emph{e.g.}, \cite{fu2008simple,GrHo,shapiro1986reconstructing,zhang2011approximate}), and many methods have been proposed  (see, \emph{e.g.}, \cite{batenkov2018stable,bertero1979problems,chen10,coifman2006geometric,drouiche2001regularization,landau1986extrapolating,miller1970least,trefethen2019quantifying}). To obtain precise error bounds, it is useful to obtain nonasymptotic upper and lower bounds on the singular values of the truncated Fourier transform 
rather than the more usual 
asymptotic estimates on a logarithmic scale.   
In several applications, uniform estimates on the right singular functions are useful as well. This occurs for example to show that certain nonparametric statistical procedures involving series are adaptive (see, \emph{e.g.},  \cite{chagny2015adaptive}). This means that an estimator with a data-driven smoothing parameter reaches the optimal minimax rate of convergence.  Importantly, such a program providing nonasymptotic bounds on the singular values and right singular functions has been carried recently in relation to bandlimited functions in \cite{bonami2018,bonami2016uniform,bonami2017approximations,Osipov1, Osipov,Xiao}. A second important aspect is the access to efficient methods to obtain the singular value decomposition (henceforth SVD). While numerical solutions to the inverse problems have for a long time relied on the Tikhonov or iterative methods such as the Landweber method (Gerchberg method for out-of band extrapolation, see \cite{Iv_imagaing}) to avoid using the SVD, recent developments have made it possible to approximate efficiently the PSWF and the SVD  (see Section 8 of \cite{Osipov}). 

Assuming that the function observed on an interval is the restriction of a bandlimited function can be questionable (see \cite{SlepianB}). For example, in the case of censored data, the observed function is a truncated density and the underlying function a density, and none of the usual families are bandlimited. Moreover, even if the function were bandlimited, one would require an upper bound on $1/b$ which might not be available in practice (see \cite{slepian1983some}). In the nonparametric random coefficients model \eqref{eq}, the bandlimited assumption amounts to a bounded vector $(\alpha,\beta^\top)$ which is not a meaningful assumption. 
For this reason, this paper considers the substantially larger class of functions whose Fourier transforms belong to the space $L^2(e^{b|\cdot|})$
of square-integrable functions with weight function $e^{b|\cdot|}$. 
This is the largest space that we can consider to extrapolate a function with Hilbert space techniques  because, for $a>0$, $\left\{f\in L^2(\R):\ \forall b<a,\ 
\mathcal{F}[f]\in L^{2}\left(e^{b|\cdot|}\right)\right\}$ is the set of square-integrable functions which have an analytic continuation on $\{z\in\C:\ |\mathrm{Im}(z)|<a/2\}$ (see Theorem IX.13 in \cite{ReedSimon}).  
The broader class of functions whose Fourier transforms belong to $L^2(e^{b|\cdot|})$ has rarely been used in this context and, unlike the PSWF, much fewer results are available, with the notable exception of \cite{morrison1962commutation,Widom}. It is considered in \cite{belitser1998efficient}  in the case of censored data and in \cite{estimation} for the problem of estimating  the density of random coefficients in the linear random coefficients model. 
There, it is related to a fundamental property of the law of the random coefficients: namely it does not have heavy-tails.  Equivalently, this means that its Laplace transform is finite near 0. 

In this paper, we use the weight $\cosh(b\cdot)$ rather than $e^{b|\cdot|}$ because the Fourier transform of $\sech=1/\cosh$ is essentially itself and, though with a different scalar product, $L^{2}\left(\cosh(b\cdot)\right)=L^{2}\left(e^{b|\cdot|}\right)$. Theorem II in \cite{Widom} provides, for given $b,c>0$ and a value of the index $m$ going to infinity,  an equivalent of the logarithm of the singular values of the truncated Fourier transform 
acting on $L^{2}\left(\cosh(b\cdot)\right)$. Such an equivalent is important but this result is silent on the polynomial preexponential factors, their dependence with respect to $c$ and $b$, and to deduce upper and lower bounds on the singular values which hold for all $m$, $c$, and $b$. 
This behavior is important in \cite{estimation} where we integrate the bounds over $c$ in intervals $[a,b]$ where $a$ can be arbitrarily close to $0$.  

This paper provides nonasymptotic upper and lower bounds on the singular values, with similar qualitative behavior, and applies the lower bounds to error bounds for stable analytic continuation using the spectral cut-off method. There, the nonasymptotic lower bounds are important to obtain a tight polynomial rates of convergence for ``supersmooth functions". 
We also analyze a differential operator which commutes with a symmetric integral operator obtained by applying the truncated Fourier operator to its adjoint. The corresponding eigenvalue problem involves singular Sturm-Liouville equations. This allows to prove uniform estimates on the right-singular functions. Working with this differential operator is useful because its eigenvalues increase quadratically while those of the integral operator decrease exponentially.  
Solving numerically singular differential equations allows to obtain these functions, hence all the SVD. 
These results are fundamental to obtain a feasible data-driven minimax estimator of the density of $(\alpha,\beta^\top)$ in \eqref{eq}  (see \cite{estimation}). However, the range of applications is much larger. For the sake of simplicity, we illustrate some of our results with the relatively simpler problem of extrapolation. We apply the lower bounds and numerical method to the problem of 
analytic continuation of possibly nonbandlimited functions whose Fourier transform belongs to $L^2(e^{b|\cdot|})$ by truncated SVD 
when the function is observed with error on an interval.
We propose an adaptive method to select the truncation. When the function is bandlimited and the researcher knows an interval which contains the bandlimits, we rely on the PSWF and efficient methods to compute the SVD. We then consider the case when the researcher does not have 
prior information on the bandlimits and questions the fact that the function can be bandlimited. This requires the developments of this paper and recent  
numerical schemes for solving singular differential equations. Extrapolation is simply used as an illustration.
Developments on alternative methods and extensions (for example with a random error or applying the canvas of \cite{Gosse,Gosse13}) require further research. 


\section{Preliminaries}\label{sec:prel}
We use $\N_0$ for the set of nonnegative natural numbers, $a\vee b$ for the maximum of $a$ and $b$, \emph{a.e.} for almost everywhere, and $f(\cdot)$ for a function $f$ of some generic argument. We denote, for $a>0$, by $L^2(-a,a)$ and $L^2(\R)$ the usual $L^2$ spaces 
of  complex-valued functions  equipped with the Hermitian inner product, for example
$ \braket{f,g}_{L^2(-a,a)} = \int_{-a}^a f(x)\overline{g}(x)dx$, by $L^2(W)$ for a positive function $W$ on $\R$ the weighted $L^2$ spaces 
equipped with $ \braket{f,g}_{L^2(W)} = \int_{\R} f(x)\overline{g}(x)W(x)dx$,  
and by $S^{\perp}$ the orthogonal complement of the set $S$ in a  Hilbert space. We denote by $\|f\|_{L^{\infty}([a,b])}$ the sup-norm of the function $f$ on $[a,b]$. 
The inverse of a mapping $f$, when it exists, is denoted by $f^I$. We denote, for $b,c>0$, by 
\begin{equation}\label{eqdef1} 
\begin{array}{cccc}
\mathcal{C}_{c}: & L^{2}\left(\R\right) & \rightarrow &  L^{2}\left(\R\right)  \\
				     &      f		      & \mapsto &   cf(c\cdot)
\end{array} ,\quad
\begin{array}{cccc}
\mathcal{F}_{b,c} : & L^{2}\left(\cosh(b\cdot)\right) & \rightarrow &  L^2(-1,1) \\
			    & f 		& \mapsto &   \mathcal{F}\left[  f \right](c\cdot)
\end{array},  
\end{equation}
 by  $\mathcal{F}\left[f\right]=\int_{\mathbb{R}}e^{ix\cdot}f(x)dx$ the Fourier transform of $f$ in $L^1\left(\mathbb{R}\right)$ and also use the notation $\mathcal{F}\left[f\right]$ for the Fourier transform in $L^2\left(\mathbb{R}\right)$. 
$\mathcal{O}^*$ denotes the Hermitian adjoint of an operator $\mathcal{O}$.
Recovering $f$ such that for $b$ small enough $f\in L^{2}\left(\cosh(b\cdot)\right)$ based on its Fourier transform on $[-c,c]$ amounts to inverting $\mathcal{F}_{b,c}$. This can be achieved using the SVD. Define the finite convolution operator
\begin{equation}\label{eq:Qc}
\begin{array}{cccc}
\mathcal{Q}_c: &L^2(-1,1) &\to& L^2(-1,1)\\
&h&\mapsto&\int_{-1}^{1} \pi c\ \sech\left(\frac{\pi c}{2}(\cdot -y)\right) h(y)dy.
\end{array}
\end{equation}
It is a trace class (see 
Theorem 2 in \cite{Widom}), symmetric, and positive operator on spaces of real and complex valued functions.
Denote by $\left(\rho_m^{c}\right)_{m\in\N_0}$ its positive real eigenvalues in decreasing order and  repeated according to multiplicity and by $\left(g_m^{c}\right)_{m\in\N_0}$ its eigenfunctions which can be taken to be real valued. 
The next proposition 
relies on the fact that, for all $c>0$, $\mathcal{F}\left[\sech(c \cdot)\right](\star)=(\pi/c)\sech(\pi\star/(2c))$.



\begin{proposition}\label{prop:Q}
For $b,c >0$, we have $c\mathcal{F}_{b,c}\mathcal{F}_{b,c}^*=  \mathcal{Q}_{c/b}$.
\end{proposition}
\begin{proof}
In this proof, $\mathcal{E}$ is the operator from $L^2(-1,1)$ to $L^{2}\left(\cosh(b\cdot)\right)$ which assigns the value 0 outside $[-1,1]$ and 
$\mathcal{R}:\ L^2(\R)\to L^2(\R)$ is such that $\mathcal{R} f=f(- \cdot)$.
Because $\mathcal{F}_{b,c}=\mathcal{F}\mathcal{C}_{c^{-1}}=c^{-1}\mathcal{C}_{c}\mathcal{F}$, $\mathcal{R}\mathcal{F}_{b,c}=\mathcal{F}_{b,c}\mathcal{R}$, 
\begin{equation}\label{eq:adjoint}
\mathcal{F}_{b,c}^*=\sech(b\cdot)\mathcal{R}\mathcal{F}_{b,c}\mathcal{E},
\end{equation} and $\sech(b\cdot)	$ is even, we obtain
$\mathcal{F}_{b,c}^*=\mathcal{R}\left[\sech(b\cdot)\mathcal{F}_{b,c}\mathcal{E}\right]$ and 
\begin{align*}
c\mathcal{F}_{b,c}\mathcal{F}_{b,c}^*&=c\mathcal{R}\mathcal{F}_{b,c}\left[\sech(b\cdot)\mathcal{F}_{b,c}\mathcal{E}\right]
=2\pi\mathcal{F}^I\left[\mathcal{C}_{c^{-1}}\left[\sech(b\cdot)\mathcal{C}_c\mathcal{F}\mathcal{E}\right]\right]
=2\pi c \mathcal{F}^I\left[\mathcal{C}_{c^{-1}}\left[\sech(b\cdot)\right]\mathcal{F}\mathcal{E}\right],
\end{align*}
where, for \emph{a.e.}	 $x\in \R$, 
\begin{align*}
2\pi c\mathcal{F}^I\left[\mathcal{C}_{c^{-1}}\left[\sech(b\cdot)\right]\right](x)&=\int_{\R}  e^{-itx} \sech\left(\dfrac{bt}{c}\right) dt= \dfrac{\pi c}{b} \sech\left(\dfrac{\pi c }{2b} x\right)  .
\end{align*}
As a result, we have, for $f\in L^2(-1,1)$,
\begin{align*}
c\mathcal{F}_{b,c}\mathcal{F}_{b,c}^*[f]&= \mathcal{C}_{\pi c/(2b)}\left[2 \sech\right]\ast \mathcal{E} [f] =\mathcal{Q}_{ c/b}[f].
\end{align*}
\end{proof}
Proposition \ref{prop:Q} yields that $(g_m^{c/b})_{m\in\N_0}$ are the right singular functions of $\mathcal{F}_{b,c}$. The SVD of $\mathcal{F}_{b,c}$, denoted by $\left(\sigma_m^{b,c},\varphi_m^{b,c},g_m^{ c/b}\right)_{m\in\N_0}$, is such that, for $m\in\N_0$, 
\begin{equation}\label{esigmarho}
\sigma_m^{b,c} = \sqrt{\frac{\rho_m^{ c/b}}{c}}\end{equation} 
 and 
 $\varphi_{{m}}^{b,c}=\mathcal{F}_{b,c}^*g_{m}^{ c/b}/\sigma_{{m}}^{b,c}$. It yields 
\begin{equation}\label{ecore}
\boxed{\quad\forall b,c>0,\ \forall f\in L^{2}\left(\cosh(b \cdot)\right),\  f=\sum_{m\in \N_0}\frac{1}{\sigma_{m}^{b,c}}\left\langle \mathcal{F}_{b,c}\left[f\right],g_{m}^{ c/b}\right\rangle_{L^2(-1,1)}\varphi_{{m}}^{b,c}.\quad}
\end{equation}
\eqref{ecore} is a core element for out-of-band extrapolation 
when the signal $f$ does not have compact support. The stepping stone for extrapolation is \eqref{stepping} in Section \ref{sec:analy}. It is important to decouple $b$ and $c$ because $b$ is a feature of the unknown function while $2c$ is a feature of the observed data, namely the window of observations of the Fourier transform. Though $c/b\simeq 1$ is a popular regime in signal processing, we do not restrict ourselves to this setting. In \cite{estimation}, $c$ can get arbitrarily large or small and $b$ is fixed and indices the class of functions.  
\begin{proposition}\label{sec:upper:extension} For all $b,c>0$, $\mathcal{F}_{b,c}$ is injective and $\left(\varphi_m^{b,c}\right)_{m\in\N_0}$  is a basis of $L^{2}(\cosh(b\cdot))$.
\end{proposition}
\begin{proof}
We use that, for every $h\in L^{2}(\cosh(b\cdot))$, if we do not restrict the argument in the definition of $\mathcal{F}_{b,c}[h]$ to $[-1,1]$, $\mathcal{F}_{b,c}[h]$ can be defined as a function in $L^2(\R)$. In what follows, for simplicity, we use $\mathcal{F}_{b,c}[h]$ for both the function in $L^2(-1,1)$ and in $L^2(\R)$.\\ 
Let us show that $\mathcal{F}_{b,c}$ defined in \eqref{eqdef1} is injective. Take $h\in L^{2}(\cosh(b\cdot))$ such that $\mathcal{F}_{b,c} [h]$ is zero on $[-1,1]$. 
Then, using Theorem IX.13 in \cite{ReedSimon}, $\mathcal{F}_{b,c}[h]$ is zero on $\R$. Thus, $\mathcal{F}[h]$ hence $h$ are zero \emph{a.e.} on $\R$.\\
The second part of Proposition  \ref{sec:upper:extension} holds for the following reasons. 
Because $\mathcal{Q}_c$ is trace class, $\mathcal{F}_{b,c}^*$ is Hilbert-Schmidt hence compact (see, \emph{e.g.}, exercise 47 page 220 in \cite{ReedSimon}). Thus $\mathcal{F}_{b,c}$ is also compact  (see Theorem VI.12 in \cite{ReedSimon}). Then, the result holds by Theorem 15.16 in \cite{Kress} and the injectivity of $\mathcal{F}_{b,c}$.
\end{proof}

Theorem II in \cite{Widom} provides  the equivalent
\begin{equation}\label{eqWidom} 
\log\left(\rho_m^{c}\right) \underset{m\to\infty}{\sim}  - \pi m \dfrac{K(\text{sech}(\pi c))}{K(\tanh(\pi c))},
\end{equation}
where $K(r) = \int_{0}^{\pi/2}\left(1 - r^2\sin(x)^2 \right)^{-1/2}dx $ is the complete elliptic integral of the first kind. This paper provides nonasymptotic upper and lower bounds on the eigenvalues and upper bounds on the sup-norm of the functions $\left(g_m^{c}\right)_{m\in\N_0}$. This is important when, as in \cite{estimation}, one needs bounds on all the singular values for a (possibly diverging) range of $c$. In contrast, \eqref{eqWidom} is an equivalent for diverging $m$ and given $c$ on a logarithmic scale. Thus, it is silent on constants and preexponential factors and how they depend on $c$.  

The proofs of this paper sometimes rely on the following operator 
\begin{equation}\label{eqdefFpSWF} 
\begin{array}{cccc}
\mathcal{F}_c^{W_{[-1,1]}} : & L^{2}\left(W_{[-1,1]}\right) & \rightarrow &  L^2(-1,1),		  \\
			    & f 		& \rightarrow &   \mathcal{F}\left[  f \right](c\cdot)
\end{array} 
\end{equation}
where $W_{[-1,1]}=\indic\left\{[-1,1]\right\}+\infty\ \indic\left\{[-1,1]^c\right\}$, for which we use the notations 
$\rho_m^{W_{[-1,1]},t_m}$ for the $m^{\rm{th}}$ eigenvalue of $\mathcal{Q}_c^{W_{[-1,1]}}=c\mathcal{F}^{W_{[-1,1]}}_{c}\left(\mathcal{F}_{c}^{W_{[-1,1]}}\right)^*$. 

\section{Lower bounds on the eigenvalues of $\mathcal{Q}_c$ and an application}
\label{sec:low}

\subsection{Lower bounds on the eigenvalues of $\mathcal{Q}_c$}


\begin{lemma}\label{it:decr_sigma}
For all $m\in\N_0$,  $ c\in (0,\infty) \mapsto \rho_{m}^{c} $ is nondecreasing.
\end{lemma}
\begin{proof} 
Take $ m\in\N_0$. Using the maximin principle (see Theorem 5 page 212 in \cite{birman2012spectral}),  the $m+1$-st  eigenvalue $\rho_m^{c}$ satisfies 
$$  \rho_m^{c} = \underset{V \in S_{m+1}}{\max} \underset{f \in V\setminus\{0\}}{\min}\dfrac{\braket{\mathcal{Q}_cf,f}_{L^2(-1,1)}}{\left\| f\right\|^2_{L^2(-1,1)}} ,$$
where $S_{m+1}$ is the set of $m+1$-dimensional vector subspaces of $L^2(-1,1)$.
Using \eqref{eq:adjoint} and Proposition \ref{prop:Q}, we obtain 
\begin{align}
 \notag \braket{\mathcal{Q}_cf,f}_{L^2(-1,1)}  
&=  c\braket{ \mathcal{F}_{1,c}\mathcal{F}_{1,c}^* [f] ,f}_{L^2(-1,1)}  \notag\\
&=  c\braket{   \mathcal{F}_{1,c}^*[f] ,\mathcal{F}_{1,c}^*[f]}_{L^2(\cosh)} \label{eq:Qc1}\\
&= c\left\|\sech \times \mathcal{F}_{1,c}\left[ \mathcal{E}\left[f\right] \right] \right\|_{L^2(\cosh)}^2 \notag\\
&=  c\int_{\R}\sech\left(x\right)\left|\int_{\R} e^{ictx} \mathcal{E}\left[f\right] (t)dt \right|^2dx \notag\\
&= \int_{\R} \sech\left(\dfrac{x}{c}\right)\left|\mathcal{F}\left[ \mathcal{E}\left[f\right]\right](x)\right|^2 dx\notag 
\end{align}
hence 
\begin{equation} \label{eq:Q}
\rho_m^{c} = \underset{V \in S_{m+1}}{\max} \underset{f \in V\setminus\{0\}}{\min}  \dfrac{2\pi\int_{\R} \sech\left(x/c\right)\left|\mathcal{F}\left[ \mathcal{E}\left[f\right]\right](x)\right|^2dx}{\left\|  \mathcal{F}\left[ \mathcal{E}\left[f\right]\right]\right\|^2_{L^2(\R)}}.
\end{equation}  
Then, using that $t\mapsto \cosh( t)$ is even, nondecreasing on $[0,\infty)$, and positive, we obtain that, for all $0<c_1\le c_2$ and $x\in\R$, 
$\sech\left(  x/c_2\right) \geq \sech\left(  x/c_1\right)$ hence that $ \rho_{m}^{c_1}\leq\rho_{m}^{c_2} $. 
\end{proof}
\begin{theorem}\label{prop:low1}
For all $m\in \N_0$, we have
\begin{align}
\forall\ 0<  c \leq \frac{\pi}{4},\ \rho_m^{c}  &\geq \dfrac{2\sin(2c)^2}{ e^2c}   \exp\left( -2\log\left(\dfrac{7e^2\pi}{2c}\right)m   \right)\label{eq:low1_1}\\
\forall c>0,\ \rho_m^{c} &\ \geq \pi  \exp\left(-\df{\pi(m+1)}{2c}\right).\label{eq:low1_2}
\end{align}
\end{theorem}

\eqref{eq:low1_1} is valid for $0<c \leq\pi/4$ and more precise  than \eqref{eq:low1_2} for $c$ close to $0$. \eqref{eq:low1_2} is uniformly valid. To prove it, we show that  
$\rho_m^{c}  \geq 	\sech\left(t_m/c\right) \rho_m^{W_{[-1,1]},t_m}$ for well chosen $t_m$ and rely on a  lower bound on $\rho_m^{W_{[-1,1]},t_m}$. The proof of \eqref{eq:low1_1} uses similar arguments as those in \cite{bonami2018} Section 5.1 and a lower bound on the best constant $\Gamma(m,\epsilon)$ such that for all interval $I\subseteq[-\pi,\pi]$ of length $2\epsilon>0$ and all polynomial of degree at most $m\in\N_0$, 
\begin{align*}
		\left\| P\left(e^{i\cdot}\right)\right\|^2_{L^2(I)} \geq \Gamma(m,\epsilon) \left\| P\left(e^{i\cdot}\right)\right\|^2_{L^2(-\pi,\pi)}.
\end{align*}
We use the lower bound in \cite{nazarov2000complete} page 240
\begin{equation}
		\Gamma(m,\epsilon) \geq \left(\frac{14e\pi}{\epsilon} \right)^{-2m}\frac{\epsilon}{\pi}\label{eq:TN}
\end{equation}
for $\epsilon=4c$. 
It is argued in \cite{nazarov2000complete} that it cannot be significantly improved for small $\epsilon$ which is precisely the regime for which \eqref{eq:low1_1} is used to bound the eigenvalues from below. 
%
\begin{proof}
 Let $m\in\N_0$, $c>0$, and $M = (m+1)/(2c)$. 
For $R>0$, we denote by $PW(R)$ the Paley-Wiener space of functions whose Fourier transform has support in $[-R,R]$ and by $S_{m+1}(R)$ the set of $m+1$-dimensional subspaces of $PW(R)$. Using \eqref{eq:Q}, we have
$$  \rho_m^{c} = \underset{V \in S_{m+1}(1)}{\max} \underset{g \in V\setminus\{0\}}{\min} \dfrac{ 2\pi\int_{\R} \sech\left(x/c\right) \left|g(x)\right|^2 dx}{\left\|  g\right\|^2_{L^2(\R)}}  . $$
Then, for $g\in PW(1)$, 	the function $g_{Mc}: \ x\in \R \mapsto (Mc)^{1/2}g(Mcx)$ satisfies $\left\|  g\right\|^2_{L^2(\R)}=\left\|  g_{Mc}\right\|^2_{L^2(\R)}$ and belongs to $PW(Mc)$. 
Using
$$\int_{\R} \sech\left(\dfrac{ x}{c}\right) \left|g(x)\right|^2 dx = \int_{\R} \sech\left(Mx\right) \left|g_{Mc}(x)\right|^2 dx, $$
we have, for $V \in S_{m+1}(Mc)$,
\begin{equation}\label{lb} 
\rho_m^{c} \geq  \underset{g \in V\setminus\{0\}}{\min} \dfrac{2 \pi \int_{\R}\sech\left(Mx\right)\left| g(x)\right|^2 dx}{\left\|  g\right\|^2_{L^2(\R)}}. 
\end{equation}
Let us now choose a convenient such space $V$ defined, for $\varphi: t\in \R \mapsto \sin(t/2)/(\pi t)$, as $$V
=\left\{
\sum_{k=0}^mP_ke^{i(k-m/2)\cdot}\varphi(\cdot), \quad (P_k)_{k=0}^m\in\C^{m+1}\right\}.$$ 
The Fourier transform of an element of $V$ is of the form  
$\sum_{k=0}^m P_k\mathcal{F}\left[ \varphi \right]\left(\cdot-k+m/2\right)$ 
and, because $\mathcal{F}\left[ \varphi \right](\cdot) = \indic\{|\cdot|\leq 1/2\}$, it has support in 
$[-1/2-m/2,1/2+m/2] = [-Mc,Mc]$. This guarantees that $V
\in S_{m+1}(Mc)$.\\ 
 We now obtain a lower bound on the right-hand side of \eqref{lb}. Let $g\in V$, defined via the coefficients $(P_k)_{k=0}^m$, and, for $x\in\R$, let $P(x)=\sum_{k=0}^mP_kx^k$. 
Let $0<x_0\leq \pi/2$. We have, using  $\forall x\in[0,2x_0)$, $\sin(x/2)/x\ge  \sin(x_0)/(2x_0)$ for the last display, 
\begin{align*}
 \int_{\R} \sech\left( Mx\right) \left|g(x)\right|^2 dx 
&\ge \int_{-2x_0}^{2x_0}\sech\left(Mx\right) \left| \sum_{k=0}^m P_k e^{ikx} \right|^2 \left| \varphi(x) \right|^2     dx \\
& \geq  \dfrac{1}{\cosh( 2Mx_0)} \min_{x \in [-2x_0, 2x_0]}\left|\varphi(x) \right|^2 \int_{-2x_0}^{2x_0} \left| \sum_{k=0}^m P_k e^{ikx} \right|^2  dx\\
& \geq \dfrac{\sin(x_0)^2}{(2\pi x_0)^2} e^{- 2Mx_0}  \left\| P\left(e^{i\cdot}\right)\right\|^2_{L^2(-2x_0,2x_0)}.
\end{align*}
Now, using that, for $k \in \N_0$, $t\mapsto \mathcal{F}\left[ \varphi \right](t-k+m/2)$ have disjoint supports,  we obtain 
\begin{align*}
\left\|g \right\|^2_{L^2(\R)} &=  \dfrac{1}{2\pi} \left\|\mathcal{F}\left[g\right] \right\|^2_{L^2(\R)}\\
&= \dfrac{1}{2\pi}\sum_{k=0}^m \left|P_k \right|^2 \left\| \mathcal{F}\left[\varphi\right]\right\|^2_{L^2(\R)}   \\
&=  \dfrac{1}{(2\pi)^2} \left\| P\left(e^{i\cdot}\right)\right\|^2_{L^2(-\pi,\pi)} ,
\end{align*} 
hence, by \eqref{eq:TN},
\begin{align*}\rho_m^{c} 
&\ge    \dfrac{4\sin(x_0)^2}{ x_0} e^{- 2Mx_0} \left(\dfrac{7e\pi }{x_0} \right)^{-2m}.
\end{align*}
We obtain, for $ 0< x_0 \leq \pi/2 $ and $m\in \N_0$,
\begin{align}\notag
\rho_m^{c} & \geq \dfrac{4 \sin(x_0)^2}{ x_0}   \exp\left( -\dfrac{ x_0}{c}(m+1) - 2\log\left(\dfrac{7e\pi}{x_0}\right)m    \right).
\end{align}
Thus, we have, for all $m\in\N_0$, 
\begin{align}\label{form}
\rho_m^{c} 
& \ge 4 e^{-2\log(7e\pi)m}	\sup_{x_0\in(0,\pi/2]} \dfrac{\sin(x_0)^2}{ x_0}e^{-x_0/c}\exp\left( -\left(\dfrac{ x_0}{c} - 2\log\left(x_0\right)\right)m   \right).
\end{align}
Using that if $2c <\pi$,  $x_0\mapsto x_0/c - 2\log\left(x_0\right)$ admits a minimum at $x_0 = 2c$, we obtain, for all $ 0< c \leq \pi/4$,
\begin{align*}\notag
\rho_m^{c} & \geq \dfrac{2\sin(2c)^2}{ e^2c}   \exp\left( -2\log\left(\dfrac{7e^2\pi}{2c}\right)m   \right).
\end{align*}
\noindent We now prove the second bound on $\rho_m^{c}$.  
Let $m\in\N_0$ and  $ t_m =\pi(m+1)/2 $. For all $ x \in \R $, we have $ \sech\left( x/c \right)\geq \sech\left(t_m/c\right) \indic\{\abs{x}\leq t_m \} $, hence, by \eqref{eq:Q}, we have 
\begin{align}
 \rho_m^{c} & =  \underset{V \in S_{m+1}}{\max} \underset{f \in V\setminus\{0\}}{\min}   \int_{\R} \sech\left(\dfrac{ x}{c}\right)\left|\mathcal{F}\left[ \mathcal{E}\left[f\right]\right](x)\right|^2 dx \dfrac{1}{\|f\|_{L^2(-1,1)}^2}\notag\\
 & \geq  \sech\left(\dfrac{t_m}{c}\right) \underset{V \in S_{m+1}}{\max} \underset{f \in V\setminus\{0\}}{\min}  \int_{\R} \indic\{\abs{x}\leq t_m \}  \left|\mathcal{F}\left[ \mathcal{E}\left[f\right]\right](x)\right|^2 dx \dfrac{1}{\|f\|_{L^2(-1,1)}^2} \notag\\
  & \geq 	\sech\left(\dfrac{t_m}{c}\right) \rho_m^{W_{[-1,1]},t_m}.\label{eq:lastW}
\end{align} 
Using that $m = 2t_m/\pi -1$ and (5.2) in \cite{bonami2018} 
 (with a difference by a factor $1/(2\pi)$ in the normalisation of $\mathcal{Q}_c^{W_{[-1,1]}}$), we have  $\rho_m^{W_{[-1,1]},t_m}  \geq \pi$ hence, for all $ m \in \N_0  $, 
\begin{align*}
 \rho_m^{c} & \geq \exp\left(-\df{t_m}{c}\right)  \rho_m^{W_{[-1,1]},t_m} \quad (\text{by} \ \eqref{eq:lastW}) \notag \\ 
 & \geq \pi \exp\left(-\df{\pi(m+1)}{2c}\right). 
\end{align*}
\end{proof}

The best lower bound in terms of the factor in the exponential is \eqref{eq:low1_1} for $c\le c_0$, where $c_0=0.12059$, 
and \eqref{eq:low1_2} for larger $c$ (see  Figure \ref{fig:bounds}). This yields 
\begin{corollary}\label{cor:X} For all $c>0$, 
\begin{align}
\forall m\in\N_0,\ &\rho_m^{c}\geq\theta(c) e^{ -2\beta(c) m}\label{eq:theta_beta},
\end{align}
where
\begin{align*}
\beta:&\ c \mapsto  \log\left(\dfrac{7e^2\pi}{2c}\right) \indic\left\{ c\leq c_0 \right\} +   \dfrac{\pi}{4c} \indic\left\{ c > c_0 \right\},\\
\theta:& \ c \mapsto \dfrac{2 \sin(2c)^2}{e^2c} \indic\left\{ c \leq  c_0\right\} + \dfrac{\pi}{ e^{	\pi/(2c)}}\indic\left\{ c > c_0\right\} .
\end{align*}
Clearly, because $c_0\le\pi/4$ and  $x\mapsto \sin(x)/x$ is decreasing on $(0,\pi/2]$, the lower bound holds when we replace $\theta$ by 
$$\widetilde{\theta}:\ c \mapsto \dfrac{2\sin(2c_0)^2c}{(ec_0)^2}  \indic\left\{ c \leq  c_0\right\} + \dfrac{\pi}{ e^{	\pi/(2c)}}\indic\left\{ c > c_0 \right\}.$$ 
\end{corollary}

\subsection{Application: Error bounds for stable analytic continuation of functions whose Fourier transform belongs to $L^2(\cosh(b\cdot))$}\label{sec:analy} In this section, we consider the problem where 
we observe the function $f$ with error on $(x_0-c,x_0+c)$, for $c>0$ and $x_0 \in \R$, 
\begin{equation}\label{eq:f}
f_{\delta}(cx + x_0) = f(cx + x_0) + \delta \xi(x),  \quad \text{for \textit{a.e.}}  \ 	x\in (-1,1),  \quad \mathcal{F}\left[f\right]\in L^2(\cosh(b \cdot)),
	\end{equation}
where $\xi \in L^2(-1,1)$, $\|\xi\|_{L^{2}(-1,1)} \leq 1$, and $\delta >0$. 
We consider the problem of approximating $f_0=f$ on $L^2(\R)$ from $f_\delta$ on $(x_0-c,x_0+c)$. This is a classical problem for which an approach based on PSWF is prone to criticism when the researcher does not have a priori information on the bandlimits or questions the bandlimited assumption. As we have stressed before such an assumption makes little sense for probability densities. 

Noting that, 
for a.e. 	$x\in (-1,1)$,
\begin{equation}\label{eq:f1}
  \dfrac{1}{2\pi}\mathcal{F}_{b,c}\left[ \mathcal{F}\left[f(x_0- \cdot)\right]\right](x) = f(cx + x_0),
\end{equation}
we have, $\forall b,c>0$, 
\begin{equation}\label{stepping}
\boxed{\quad\forall f:\ \mathcal{F}\left[f\right]\in L^2(\cosh(b \cdot)),\  f=2\pi \mathcal{F}^{I}\left[ \sum_{m\in \N_0}
\frac{1}{\sigma_{m}^{b,c}}\left\langle f(c\cdot + x_0),g_{m}^{ c/b}(\cdot)\right\rangle_{L^2(-1,1)}\varphi_{{m}}^{b,c}\right].}
\end{equation}

It is classical in inverse problems that \eqref{stepping} cannot be used. Moreover, the problem is severely ill-posed and the observations are contaminated with error. In order to stabilize the inversion, we use a truncated series. 
This suggests the two steps regularising procedure:
\begin{enumerate}
\item approximate $ \mathcal{F}\left[f(x_0- \cdot)\right]/(2\pi) \in L^2(\cosh(b \cdot)) $ by the spectral cut-off regularization,
\begin{equation}\label{eq:ft1}  F^{N}_{\delta} =  \sum_{m \leq N}\frac{1}{\sigma_{m}^{b,c}}\left\langle f_{\delta}(c\cdot + x_0),g_{m}^{ c/b}(\cdot)\right\rangle_{L^2(-1,1)}\varphi_{{m}}^{b,c},
\end{equation}
\item take the inverse Fourier transform and define 
\begin{equation}\label{eq:ft2}
f^{N}_{\delta}(\cdot) =2\pi \mathcal{F}^{I}\left[ F^{N}_{\delta}\right](x_0- \cdot).
\end{equation}
\end{enumerate}
These steps require numerical approximations of an inner product, of an inverse Fourier transform over $\R$, and of the singular functions. 
Sections \ref{sec:nmerical} and \ref{sec:sim} address these issues. 

The lower bounds 
of Theorem \ref{prop:low1} 
are useful to obtain rates of convergence when 
$\mathcal{F}[f]$, which appears on the left-hand side of \eqref{eq:f1}, satisfies a source condition: $f\in\mathcal{H}_{\omega, x_0}^{b,c}(M)$, where 
\begin{align*}
\mathcal{H}_{\omega, x_0}^{b,c}(M) = \left\{ f: \quad 			\sum_{m\in \N_0} \omega_m^2	 \left|\left< \mathcal{F}[f(x_0 - \cdot)],\varphi_m^{b,c}\right>_{L^2(\cosh(b\cdot))}\right|^2\leq M^2 \right\}
\end{align*}
for a given sequence $\left(\omega_m\right)_{m\in\N_0}$. The set can also be written as 
\begin{align*}
\mathcal{H}_{\omega, x_0}^{b,c}(M) = \left\{ f: \quad 			\sum_{m\in \N_0} \left(2\pi\frac{\omega_m}{\sigma_m^{b,c}}\right)^2 \left|\left<f(c\cdot +x_0),g_m^{b,c}\right>_{L^2(-1,1)}\right|^2\leq  M^2 \right\}.
\end{align*}
This amounts to smoothness of $f(c\cdot +x_0)$ on $(-1,1)$. When $\omega_m=1$ for all $m$ this corresponds to analyticity of $f$ in $\R$. 
 We consider below cases where we have a preexponential polynomial and exponential sequence $\omega_m$. Theorem 1 in \cite{bonami2017approximations} provides a comparison between the smoothness in terms of a summability condition involving the coefficients on the PSWF basis and Sobolev smoothness on $(-1,1)$. Such a result is not available when the PSWF basis is replaced by $(g_m^{b,c})_{m\in\N_0}$ and requires further investigation.                                                      

\begin{theorem}\label{th:analytic}
Take $M>0$ and define $\beta$ as in \eqref{eq:theta_beta}, then we have
\begin{enumerate}
\item\label{it:1} for $(\omega_m)_{m\in \N_0} = (m^{\sigma})_{m\in\N_0}$, $\sigma > 1/2 $, $N =\left\lfloor  \overline{N} \right\rfloor$, and $\overline{N} = \ln(1/\delta)/(2\beta(c/b))$, 
\begin{equation}\label{eq:log_rate}
\sup_{f \in \mathcal{H}_{\omega, x_0}^{b,c}(M) ,\|\xi\|_{L^2(-1,1)}\leq 1}\left\| f^{N}_{\delta} - f\right\|_{L^2(\R)} = \underset{\delta \to 0}{O}((-\log(\delta))^{-\sigma}),
\end{equation}
\item\label{it:2} for $(\omega_m)_{m\in \N_0} = (e^{\kappa m})_{m\in\N_0}$, $\kappa > 0$, $N =\left\lfloor  \overline{N} \right\rfloor$, and $\overline{N} = \ln(1/\delta)/(\kappa + \beta(c/b))$,  
\begin{equation}\label{eq:exp_rate}
\sup_{f \in \mathcal{H}_{\omega, x_0}^{b,c}(M) ,\|\xi\|_{L^2(-1,1)}\leq 1}\left\| f^{N}_{\delta}- f\right\|_{L^2(\R)} =	 \underset{\delta \to 0}{O}\left(\delta^{ \kappa /\left(\kappa+\beta(c/b)\right) }\right) .	
\end{equation}
\end{enumerate}
\end{theorem}
\begin{proof}
We have, using the Plancherel equality for the first equality, 
\begin{align}
\left\| f^{N}_{\delta} - f\right\|_{L^2(\R)}^2 =& \dfrac{1}{2\pi} \left\| \mathcal{F}\left[f^{N}_{\delta}\right] -\mathcal{F}\left[ f\right]\right\|_{L^2(\R)}^2 \notag  \\
 =& \dfrac{1}{2\pi} \left\| \mathcal{F}\left[f^{N}_{\delta}(x_0-\cdot)\right] -\mathcal{F}\left[ f(x_0-\cdot)\right]\right\|_{L^2(\R)}^2 \notag  \\
 \leq  & \dfrac{1}{2\pi} \left\| \mathcal{F}\left[f^{N}_{\delta}(x_0-\cdot)\right] -\mathcal{F}\left[ f(x_0-\cdot)\right]\right\|_{L^2(\cosh(b\cdot))}^2 \notag \\
 \leq &  \dfrac{1}{\pi}  \left\| \mathcal{F}\left[f^{N}_{\delta}(x_0-\cdot)\right] -\mathcal{F}\left[ f^{N}_{0}(x_0-\cdot)\right]\right\|_{L^2(\cosh(b\cdot))}^2 
 \notag \\
& +  \dfrac{1}{\pi} \left\| \mathcal{F}\left[f^{N}_{0}(x_0-\cdot)\right] -\mathcal{F}\left[ f(x_0-\cdot)	\right]\right\|_{L^2(\cosh(b\cdot))}^2 . \label{eq:cent}
\end{align}
Using \eqref{eq:ft1} for the first equality, the Cauchy-Schwarz inequality and \eqref{esigmarho} for the first inequality, and \eqref{eq:theta_beta} 
for the second inequality, we obtain 
\begin{align}
&\left\| \mathcal{F}\left[f^{N}_{\delta}(x_0-\cdot)\right] -\mathcal{F}\left[ f^{N}_{0}(x_0-\cdot)	\right]\right\|_{L^2(\cosh(b\cdot))}^2  \notag \\
& =  \left\| \sum_{m \leq N} \frac{2\pi }{\sigma_{m}^{b,c}} \left\langle \left(f_{\delta}- f\right)(c\cdot + x_0),g_{m}^{ c/b}(\cdot)\right\rangle_{L^2(-1,1)} \varphi_m^{b,c}(\cdot)\right\|_{L^2(\cosh(b\cdot))}^2\notag\\
& =  \sum_{m \leq N} \left(\frac{2\pi }{\sigma_{m}^{b,c}}\right)^2 \left| \left\langle \left(f_{\delta}- f\right)(c\cdot + x_0),g_{m}^{ c/b}(\cdot)\right\rangle_{L^2(-1,1)} \right|^2 \notag  \\
& \leq  (2\pi)^2  \left\|\left(f_{\delta}- f\right)(c\cdot + x_0)\right\|^2_{L^2(-1,1)}   \sum_{m \leq N}\frac{c}{\rho_{m}	^{c/b}}\notag  \\
& \leq   \frac{(2\pi)^2 c \delta^2}{\theta(c)} \left\| \xi\right\|^2_{L^2(-1,1)}   \sum_{m \leq N}e^{2\beta(c/b) m}\notag \\
& \leq  \dfrac{(2\pi)^2  c \delta^2}{ \theta(c)\left(1 - e^{-2\beta(c/b)}\right)} e^{2\beta(c/b) N} \label{laformule} .
\end{align}
Using \eqref{eq:ft2}, we have 
\begin{align*}
 \mathcal{F}\left[f^{N}_{0}(x_0-\cdot)\right](\star)& = \sum_{m\leq N} \dfrac{ 2\pi}{ \sigma_m^{b,c}} \left\langle f(c\cdot + x_0),g_{m}^{ c/b}(\cdot)\right\rangle_{L^2(-1,1)} \varphi_{{m}}^{b,c}(\star)\\
 & =  \sum_{m\leq N}  \dfrac{ 2\pi}{ \sigma_m^{b,c}}  \left\langle \mathcal{F}_{b,c}\left[ \dfrac{1}{2\pi}\mathcal{F}\left[f(x_0- \cdot)\right]\right],g_{m}^{ c/b}	\right\rangle_{L^2(-1,1)} \varphi_{{m}}^{b,c}(\star)\\
& = \sum_{m\leq N} \dfrac{1}{ \sigma_m^{b,c}}  \left\langle  \mathcal{F}\left[f(x_0- \cdot)\right],\mathcal{F}_{b,c}^*\left[ g_{m}^{ c/b}\right]\right\rangle_{L^2(\cosh(b\cdot))} \varphi_{{m}}^{b,c}(\star)\\
& =  \sum_{m\leq N} \left\langle  \mathcal{F}\left[f(x_0- \cdot)\right],\varphi_{m}^{ c/b}\right\rangle_{L^2(\cosh(b\cdot))} \varphi_{{m}}^{b,c}(\star).
\end{align*}
Thus,  using Proposition \ref{sec:upper:extension} and Pythagoras' theorem, we obtain
\begin{align}
\left\| \mathcal{F}\left[f^{N}_{0}(x_0-\cdot)\right] -\mathcal{F}\left[ f(x_0-\cdot)\right]\right\|_{L^2(\cosh(b\cdot))}^2  
& =\sum_{m >N}\left|\left< \mathcal{F}[f(x_0-\cdot)],\varphi_m^{b,c}(\cdot)\right>_{L^2(\cosh(b\cdot))}\right|^2 \notag \\ 
& \leq \sum_{m \in \N_0} \left( \frac{\omega_m}{\omega_N}\right)^2 \left|\left< \mathcal{F}[f(x_0-\cdot)],\varphi_m^{b,c}(\cdot)\right>_{L^2(\cosh(b\cdot))}\right|^2\notag \\ 
& \leq \dfrac{ M^2}{\omega_N^{2} }  \quad (\text{using} \ f\in \mathcal{H}_{\omega, x_0}^{b,c}(M))\label{eq:bias}.
\end{align}
Finally, using \eqref{eq:cent}-\eqref{eq:bias} yields
\begin{equation}\label{eq:final}
\left\| f_{\delta}^{N} - f\right\|_{L^2(\R)}^2 \leq \dfrac{1}{\pi}\left( \dfrac{  (2\pi)^2  c }{\theta(c)\left(1 - e^{-2\beta(c/b)}\right)}  \delta^2 e^{2\beta(c/b) N}+  \dfrac{M^2}{\omega_N^{2}} \right).
\end{equation}
Consider case \eqref{it:1}. Take $\delta$ small enough so that $\overline{N} \geq 2$ and $\log\left( \delta \log\left(1/\delta \right)^{2\sigma}\right)\le0$. By \eqref{eq:final} and the definition of $\left(\omega_N\right)_{n\in\N_0}$ in the first display below, $\overline{N}-1 \leq N \leq \overline{N}$ in the second display, and $\overline{N} \geq 2$ in the third display, we obtain
\begin{align}\notag
\left\| f_{\delta}^{N} - f\right\|_{L^2(\R)}^2 & \leq \dfrac{N^{-2\sigma}}{\pi}\left( \dfrac{(2\pi)^2 c }{\theta(c)\left(1 - e^{-2\beta(c/b)}\right)}  \delta^2 e^{2\beta(c/b) N}N^{2\sigma}+  M^2\right) \\ 
& \leq \dfrac{\overline{N}^{-2\sigma}\left(1-1/\overline{N}\right)^{-2\sigma}}{\pi}\left( \dfrac{(2\pi)^2  c }{\theta(c)\left(1 - e^{-2\beta(c/b)}\right)}  \delta^2 e^{2\beta(c/b) \overline{N}}\overline{N}^{2\sigma}+  M^2\right) \notag\\
& \leq \dfrac{\overline{N}^{-2\sigma}2^{2\sigma}}{\pi}\left( \dfrac{(2\pi)^2 c }{\theta(c)\left(1 - e^{-2\beta(c/b)}\right)}  \delta^2 e^{2\beta(c/b) \overline{N}}\overline{N}^{2\sigma}+  M^2\right).\notag
\end{align}
Using that 
\begin{align*}
 \delta^2 \exp\left( 2\beta\left(\dfrac{c}{b}\right)  \overline{N} \right) \overline{N}^{2\sigma}
 & = \exp\left(   2\sigma \log\left( \dfrac{1}{2\beta(c/b)}\right) + \log\left( \delta \log\left(\dfrac{1}{\delta} \right)^{2\sigma} \right) \right)  \leq \left( \dfrac{1}{\beta(c/b)} 
 \right)^{2\sigma},
\end{align*}
yields
\begin{align}
\left\| f_{\delta}^{N} - f\right\|_{L^2(\R)}^2 
& \leq \dfrac{1}{\pi }  \left( 4\beta\left(\dfrac{c}{b}\right)\right)^{2\sigma} \left( \dfrac{(2\pi)^2 c }{\theta(c)\left(1 - e^{-2\beta(c/b)}\right)} \left( \dfrac{1}{\beta(c/b)} \dfrac{ \sigma}{e} \right)^{2\sigma} +  M^2 \right) 	(-\log(\delta))^{-2\sigma}, \label{eq:rate1}
\end{align}
hence the result.\\
Consider now case \eqref{it:2}.  Using 
$\overline{N} -1 \leq N \leq \overline{N}$ in the first display and $
\delta^2 \exp\left( 2\left(\beta\left(\dfrac{c}{b}\right) + \kappa \right) \overline{N} \right)=1$ and the definition of $\overline{N}$ 
in the second display,  yields
\begin{align}\notag
\left\| f_{\delta}^{N} - f\right\|_{L^2(\R)}^2 & \leq \dfrac{e^{-2\kappa \left(\overline{N} -1\right)} }{\pi}\left( \dfrac{(2\pi)^2 c }{\theta(c)\left(1 - e^{-2\beta(c/b)}\right)}  \delta^2 e^{2(\beta(c/b) + \kappa) \overline{N}}+  M^2\right) \\ 
& \leq \dfrac{e^{2\kappa}}{\pi}\left( \dfrac{(2\pi)^2  c }{\theta(c)\left(1 - e^{-2\beta(c/b)}\right)} +  M^2\right) \delta^{2\kappa/(\kappa+\beta(c/b))} ,
\notag 
\end{align}
hence the result.
\end{proof}

The rate in \eqref{eq:log_rate} does not depend on $c$ but the constant blows up as $c\to0$ (see \eqref{eq:rate1}). 
In contrast, the rate in \eqref{eq:exp_rate}  deteriorates for small values of $c$. The result \eqref{eq:exp_rate} is related to those obtained for the so-called ``2exp-severely ill-posed problems" (see \cite{cavalier2004block} for a survey and \cite{tsybakov2000best} which obtains similar polynomial rates) where the singular values decay exponentially and the functions are supersmooth. 

The proof of Theorem \ref{th:analytic} requires an upper bound on a sum involving the singular values for small $m$ in the denominator. Theorem \ref{prop:low1} allows to obtain \eqref{laformule}. 
Without it, one could at best obtain, instead of \eqref{laformule}, the upper bound 
$(2\pi)^2  c \delta^2(N+1)/\rho_N^{b,c}$. Also, because \eqref{eqWidom} is an equivalent of the logarithm we are unable to obtain a polynomial rate of convergence as sharp as in \eqref{eq:exp_rate}.

Having a finite number of terms in \eqref{eq:ft1} is what makes the method ``stable". The number of terms plays the same role as smoothing parameters in a Tikhonov or Landweber method (Gerchberg method for out-of-band extrapolation, see \cite{Iv_imagaing}). The interested reader can prove rate of convergence for these methods using Proposition \ref{prop:low1}. These alternative methods sometimes have computational advantages over methods involving the computation of the SVD. However, they can face qualification issues with rates of convergence which are not fast enough. They can also be hard to tune properly other than by applying rules of thumb.  

As we have seen, the case where $\omega_m=1$ for all $m$ corresponds to analyticity. If the unknown function has no additional smoothness, we cannot obtain a $f^N_{\delta}$ which converges. In statistical terms we simply have nonparametric identification. The only difference with a statistical problem here is that the error is assumed to be bounded rather than random. It is well known in statistics that only smooth functions can be estimated and the rates are faster as the function is smoother. Theorem \ref{th:analytic} considers the whole picture with both ordinary smooth and super smooth functions. It shows that one can have slow or fast rates of convergence. 
Taking a maximum over a class $\mathcal{H}_{\omega, x_0}^{b,c}(M)$ in Theorem \ref{th:analytic} is important to have a valid concept of optimality. 
A detailed analysis of extrapolation with a random error will be carried elsewhere and will contain a minimax lower bound which gives the best achievable rate over all possible methods. This is typically obtained as in \cite{estimation} with the results of Section \ref{sec:upper}.

Also, all regularization methods are very sensitive to the choice of the smoothing parameter and the optimal choice is usually unfeasible because it depends on the unknown function. The class to which the true function belongs is unknown in nearly all applications. For this reason, the next step to make a method useful in practice is to prove that a feasible choice of the smoothing parameter (here a $\overline{N}$ which does not depend on the class) nearly achieves the optimal rate of convergence. Similar to \cite{estimation}, we use a Goldenshkluger-Lepski type method that we present in Section \ref{sec:sim}. The proof that a nearly optimal rate is attained with such a method, up to now, always requires upper bounds on the singular functions like the ones we prove in Section \ref{sec:unif} (see, \emph{e.g.}, \cite{barron1999risk}).

\section{Upper bounds on the eigenvalues of $\mathcal{Q}_c$}\label{sec:upper}
\begin{theorem}\label{prop:sup1}
For $m\in \N_0$ and $0<c<1$, we have
$$\rho_m^{c}   \leq 	\dfrac{2\sqrt{\pi}c^{2m+1}}{\sqrt{m+3/4}(1- 	c^2)}.$$
\end{theorem}
\begin{proof} 
The proof is similar to that of Theorem 3.1 in \cite{bonami2018} so we will be brief on the common arguments.  
Let  $S_m$ be the set of $m$-dimensional vector subspaces of $L^2(-1,1)$, $\left(P_m\right)_{m\in \N_0}$ the Legendre polynomials with normalization $P_m(1)=1$, and $V$  the vector space spanned by $P_0,\dots, P_{m-1}$. 
By the minimax principle (Theorem 4	 p 212 in \cite{birman2012spectral}) and \eqref{eq:Qc1}, 
\begin{equation}\label{erho} \rho_m^{c}
= \underset{V\in S_m}{\min} \underset{f \in V^{\perp}}{\max}\dfrac{ c\left\|  \mathcal{F}_{1,c}^*[f] \right\|_{L^2(\cosh)}^2}{\left\| f\right\|^2_{L^2(-1,1)}}. 
\end{equation}
 Take $f\in V^{\perp}$ of norm 1. Using that $(\sqrt{m+1/2}P_m)_{m\in \N_0}$ is an orthonormal basis of $L^2(-1,1)$, 
 by \eqref{erho} and the Cauchy-Schwarz inequality 
\begin{align}\label{ubrho}
\rho_m^{c} &\leq c	\sum_{k=m}^{\infty}\left(k+\df{1}{2}\right) \left\|\mathcal{F}_{1,c}^* P_k\right\|^2_{L^2(\cosh)}.
\end{align}
By (18.17.19) in \cite{olver2010nist}, 
for \emph{a.e.}	 $x$ and all $c>0$, 
\begin{align*}
\mathcal{F}_{1,c}^*\left[P_k\right](x)
& =\sech\left(x\right) i^{k} \sqrt{\dfrac{2\pi}{c\abs{x}}}J_{k+1/2}(c\abs{x}),
\end{align*}
where $J_{k+1/2}$ is the Bessel function of the first kind of order $k+1/2$. By (9.1.62) in \cite{abramowitz1965handbook}, we have 
\begin{align*}
\left| \mathcal{F}_{1,c}^*\left[P_k\right](x)\right|^2
& \leq  \pi \left(\dfrac{\sech\left(x\right)}{\Gamma(k+3/2)}\right)^2\left(\frac{cx}{2}\right)^{2k}
\end{align*}
and conclude by \eqref{ubrho} and
\begin{align*}
\left\| \mathcal{F}_{1,c}^*\left[P_k\right]\right\|_{L^2(\cosh)}^2
& \leq  \pi \dfrac{1}{\Gamma(k+3/2)^2}\left(\frac{c}{2}\right)^{2k} \int_{\R}x^{2k}\sech\left(x\right)dx\\
& <  2\pi \dfrac{\Gamma(2k+1)}{\Gamma(k+3/2)^2}\left(\frac{c}{2}\right)^{2k}\\
& =  2\pi \dfrac{\Gamma(2(k+1))}{(2k+1)\Gamma(k+3/2)^2}\left(\frac{c}{2}\right)^{2k}\\
& = 2\sqrt{\pi} \dfrac{\Gamma(k+1)}{(k+1/2)\Gamma(k+3/2)}c^{2k}\quad\text{(by Legendre's duplication formula)}\\
& <  \dfrac{2\sqrt{\pi}}{(k+1/2)\sqrt{k+3/4}
}c^{2k}\quad\text{(by Kershaw's inequality, see \cite{Kershaw})}.
\end{align*}
\end{proof}
Lemma B.4 in \cite{estimation} gives upper bounds in the case of the PSWF which are uniform in $m$ and $c$. In \cite{bonami2018} the bounds are valid for all values of $c$ but only for $m$ large enough depending on $c$. The 
proof techniques do not allow to extend to $c\ge 1$ in our case 
due to the weighted integral over the line.  
Uniformity in $m$ is used to prove the minimax lower bound in  \cite{estimation}. The range of $c$ in Theorem  \ref{prop:sup1} is the important one to construct the so-called test functions to prove the minimax lower bounds in \cite{estimation}. 

By Corollary \ref{cor:X} and Theorem \ref{prop:sup1}, we have, for all $0<c\le c_0<1$ and $m\in\N$, 
$$\dfrac{2\sin(2c_0)^2c}{(ec_0)^2}\exp\left( -2 \left(\log\left(\dfrac{1}{c}\right)
+2 +\log\left(\dfrac{7\pi}{2}\right)
\right)m\right) \leq \rho_m^{c} 
  \leq 	\dfrac{2\sqrt{\pi}c_0}{\sqrt{m+3/4}(1- c_0^2)} \exp\left( -2 \log\left(\dfrac{1}{c}\right)m\right) .$$
The exponential factors in these upper and lower bounds have a similar behavior as $c$ approaches 0. In Figure \ref{fig:bounds} we compare the upper and lower bounds to Widom's equivalent \eqref{eqWidom}.
\begin{figure}[H]
	\centering
\subfigure[$c <1$]{\includegraphics[width=0.45\linewidth, height=0.3\textheight]{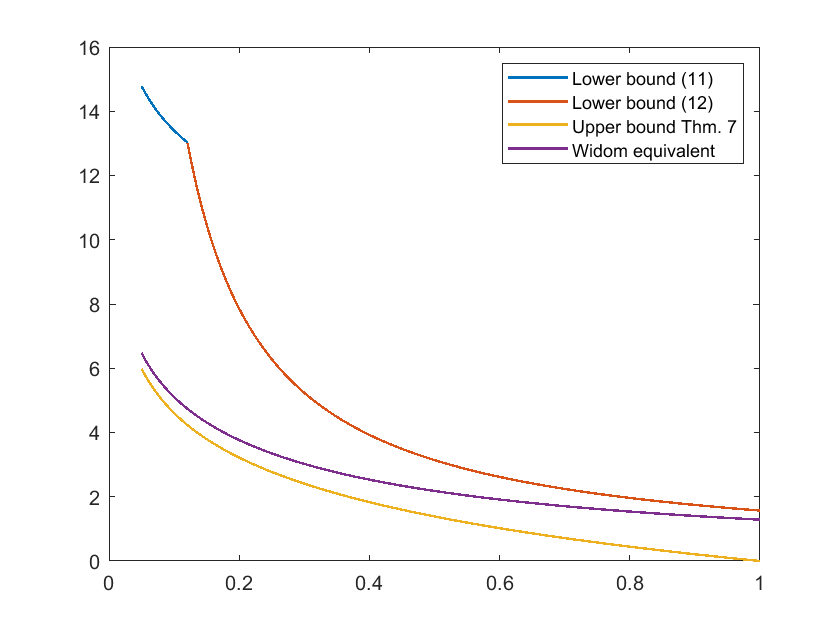}
\label{fig:subfigure1}}
\quad
\subfigure[$c \geq 1$]{\includegraphics[width=0.45\linewidth, height=0.3\textheight]{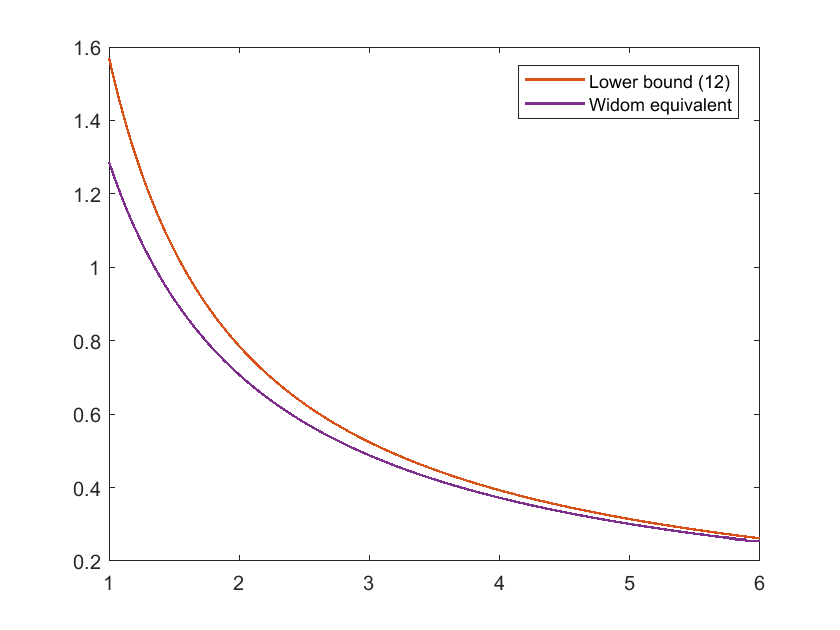}
\label{fig:subfigure2}}
\caption{
Bounds on $\lim_{m \to \infty} - \log(\rho_m^c)/m$ and Widom's equivalent \eqref{eqWidom}.}
	\label{fig:bounds}
\end{figure}

\section{Properties of a differential operator which commutes with $\mathcal{Q}_c$}\label{sec:diff}
In this section, we consider differential operators $\mathcal{L}[\psi]=-\left(p \psi'\right)'  +q\psi$ on $L^2(-1,1)$, with (1) $p (x)=\cosh(4c) - \cosh(4cx)$ and $q(x)=3 c^2\cosh(4cx)$, 
(2) $p(x)=1-x^2$ and $q(x)=q^c(x)$ where, for  $Y(x)= \sin\left(X(x)\right)$, $X(x)=\left(\pi/U(c)\right) \int_{0}^x p(\xi)^{-1/2}d\xi$,  and $$U(c)=\int_{-1}^1  p(\xi)^{-1/2}d\xi,$$ 
\begin{equation}\label{eq:tan}q^c(Y(x))=\frac12+\frac14\left(\tan\left(X(x)\right)\right)^2-\left(\frac{U(c)c}{\pi}\right)^2\left( \cosh(4cx)+\frac{\left(\sinh(4cx)\right)^2}{p(x)}\right),
\end{equation}
and (3) $p(x)=1-x^2$ and $q(x)=0$.
By \cite{Widom} (see also \cite{morrison1962commutation}), the eigenfunctions of $\mathcal{Q}_c$ are those of the differential operator in case (1) with domain $\mathcal{D} \subset \mathcal{D}_{\max} = \left\{\psi \in L^2(-1,1):\ \mathcal{L}[\psi] \in L^2(-1,1)  \right\}$ with boundary conditions of continuity at $\pm1$. This is an important property for the asymptotic analysis in \cite{Widom} and to obtain bounds on the sup-norm of these functions in Section \ref{sec:unif} and numerical approximations of them in Section \ref{sec:nmerical}. To study $\mathcal{L}$ in case (1),  \cite{Widom} uses the changes of variable and function, for all $x \in (-1,1)$ and $\psi \in \mathcal{D}_{\max}$, 
\begin{align}
y &=  Y(x)
, \   \label{eq:y}\\
\forall y\in (-1,1), \   \Gamma(y)&= F(y)\psi\left(Y^{-1}(y)\right), \  F(y)= \left(\df{p( Y^{-1}(y) )}{1-y^2}\right)^{1/4},	\label{eq:gamma}
\end{align}
where $Y$ is a $C^\infty-$diffeomorphism on $(-1,1)$. This relates an eigenvalue problem for (1) to an  eigenvalue problem for (2) and it is useful to view the operator in case (2) as a perturbation of the operator in case (3).  In the three cases, $1/p$ and $q$ are holomorphic on a simply connected open set $(-1,1)\subseteq E\subseteq\C$. The spectral analysis involves the solutions to ($H_\lambda$): $-\left(p \psi'\right)'  +(q-\lambda)\psi=0$ with $\lambda\in\C$ 
which are holomorphic on $E$ and span a vector space of dimension 2 (see Sections IV 1 and 10 in  \cite{Hartman}). So they are infinitely differentiable on $(-1,1)$, have isolated zeros in $(-1,1)$, and the condition of continuity (or boundedness) at $\pm1$ makes sense. 

We now present a few useful estimates. 
\begin{lemma}\label{lem:upper_L}
We have, for all $ c >0 $, 
$$
 \dfrac{\sqrt{2}e^{2c}}{\sinh(4c)} < U(c) <  \pi\dfrac{\sqrt{2}e^{2c}}{\sinh(4c)}. 
$$
\end{lemma}
\begin{proof}
By the second equation page 229 of \cite{Widom}
\[ U(c) = \df{1}{c(1+\cosh(4c))^{1/2}} K\left(\df{e^{4c}-1}{e^{4c}+1}\right),\]
and the result follows from the fact that, by 
Corollary 3.3 in \cite{anderson1992functional}, 
$$  \dfrac{c e^{2c}}{\sinh(2c)}      < K\left(\df{e^{4c}-1}{e^{4c}+1}\right) <  \dfrac{\pi c e^{2c}}{\sinh(2c)}.$$ 
We obtain the final expressions by classical relations between hyperbolic functions.
\end{proof}

We make use of the identity, for all $x\in[-1,1]$, 
\begin{align}
p(x) &= 4c \sinh(4c) (1- x) (1+ u(x)), \label{eq:psiu}\\ 
u(x) &=  \int_{x}^1 \dfrac{4c \cosh(4ct)}{\sinh(4c)(1-x)}(x-t) dt, \label{eq:psiu1}
\end{align}
which is obtained by Taylor's theorem with remainder in integral form
$$ \cosh(4cx) = \cosh(4c) + (x-1)4c \sinh(4c) + \int_{1}^x 16c^2 \cosh(4ct)(x-t) dt.$$
Also, $u$ is increasing on $[-1,1]$ because, for all $x\in[-1,1]$,
\begin{equation}\label{eq:u}
u'(x) =   \dfrac{4c}{\sinh(4c)(1-x)^2} \int_{x}^1 \cosh(4ct) (1-t) dt >0
\end{equation}
and, for all $x\in[0,1]$, \
\begin{equation}\label{eq:u0}
-1 + \frac{1}{4c \sinh(4c)}\left(\cosh(4c)-1\right)\le u(x)\le 0.
\end{equation}
\begin{lemma}\label{lem:lb}
We have, for all $ c >0 $ and $x\in[0,1]$, 
$$
\dfrac{c\sinh(4c)}{1-x} - 
\frac{8c^3\sinh(4c)\cosh(4c)}{3\left(\cosh(4c)-1\right)}\le \left(\int_{x}^{1}p(\xi)^{-1/2}d\xi\right)^{-2} 
\leq  \dfrac{c\sinh(4c)}{1-x} .
$$
\end{lemma}
\begin{proof}
We have 
\begin{align}
\left(\int_{x}^{1} p(\xi)^{-1/2} d\xi \right)^{-2} &=  \left( \dfrac{1}{(4c\sinh(4c))^{1/2}}\left(2(1-x)^{1/2}    - \int_{x}^{1}  \int_{1}^{\xi}  \dfrac{u'(t)}{2\sqrt{1-\xi}(1+u(t))^{3/2}}   d\xi dt\right) \right)^{-2}  \notag\\
& = \dfrac{c\sinh(4c)}{1-x} \dfrac{1}{(1+\widetilde{u}(x))^2} \label{eq:step11}\\
& = \dfrac{c\sinh(4c)}{1-x} - \dfrac{c\sinh(4c) (2+  \widetilde{u}(x)) \widetilde{u}(x) }{(1-x)(1+  \widetilde{u}(x))^2} , \label{eq:step1}
\end{align}
where 
\begin{equation}\label{e0} 
\widetilde{u}(x)=  \int_{x}^{1}  \dfrac{1}{4\sqrt{(1-\xi)(1-x)}} \left(\int_{\xi}^1   \dfrac{u'(t) }{(1+u(t))^{3/2}} dt \right) d\xi.
\end{equation}
The upper bound in Lemma \ref{lem:lb} uses that, for all $ x \in[0,1] $, $\widetilde{u}(x)\geq 0$. We now consider the lower bound. 
By 
\eqref{eq:u}, $u$ is a $C^1$-diffeomorphism and 
\begin{align}
\int_{x}^1 \dfrac{u'(t)dt}{(1+u(t))^{2}}dt 
&= - \dfrac{u(x)}{1+u(x)}. 
\label{eq:tool0}
\end{align}
Now, by \eqref{eq:psiu1}, we have, for all $x\in[0,1]$, 
\begin{align*}
- u(x) & \leq  \dfrac{4c\cosh(4c)}{\sinh(4c)(1-x)}\int_{x}^1 (t-x)dt \notag\\
&= \dfrac{2c\cosh(4c)}{\sinh(4c)}(1-x),\notag&
\end{align*}
and, by \eqref{eq:u0}, 
\begin{align}
\int_{x}^1 \dfrac{u'(t)dt}{(1+u(t))^{2}}dt &\leq 8c^2\frac{\cosh(4c)}{\cosh(4c)-1}(1-x). \label{eq:tool}
\end{align}
Now, using that $\widetilde{u}(x)\geq 0$ and that $g:\ t\mapsto (2+t)/(1+t)^2$ is decreasing on $[0,\infty	)$ hence $ g(t)t \leq 2t$ for $t\ge0$, we have 
\begin{align}
\dfrac{(2+  \widetilde{u}(x)) \widetilde{u}(x) }{(1+  \widetilde{u}(x))^2} 
& \leq   \int_{x}^{1}  \dfrac{1}{2\sqrt{(1-\xi)(1-x)}} \left(\int_{\xi}^1   \dfrac{u'(t) }{(1+u(t))^{2}} dt \right) d\xi \quad (\text{by} \  \eqref{e0}) \notag\\
& \leq  \frac{4c^2\cosh(4c)}{\cosh(4c)-1} \int_{x}^{1}  \sqrt{\dfrac{1-\xi}{1-x}}d\xi  \quad (\text{by} \ \eqref{eq:tool}) \notag\\
& \leq\frac{8c^2\cosh(4c)}{3\left(\cosh(4c)-1\right)}(1-x).\label{eq:utilde1}
\end{align}
\end{proof}
\begin{lemma}\label{lemmaF}
$F$ is such that
\begin{align}
\|F\|_{L^{\infty}([-1,1])}^4&\le2\pi^2 e^{4c} c^2 
\label{eq:F_1_bound}\\
\|1/F\|_{L^{\infty}([-1,1])}^4&\le \frac{
\pi^2e^{-4c}}{4c}\left(1+\frac{4c^2}{3}\right)^2  \coth(2c)
\label{eq:F_bound}.
\end{align}
\end{lemma}
\begin{proof}
To prove \eqref{eq:F_1_bound} and \eqref{eq:F_bound} it is sufficient, by parity, to consider $x\in[0,1]$.\\
\eqref{eq:F_1_bound} is a obtained by the following sequence of inequalities
\begin{align}
\df{p(x)}{1-Y(x)^2} =& \df{p(x)}{\left(\sin\left(\pi \int_{x}^1p(\xi)^{-1/2} d\xi /U(c)\right)\right)^2}\notag  \\ 
\leq &  \left(\df{U(c)}{2}\right)^2 \df{p(x)}{\left(\int_{x}^1p(\xi)^{-1/2} d\xi \right)^2} \quad (\text{because} \ \sin(x) \geq 2 x/\pi) \notag \\	
\leq &  \left(\df{U(c)}{2}\right)^2 p(x) \dfrac{c\sinh(4c)}{1-x} \quad (\text{by Lemma \ref{lem:lb}})\notag  \\
\leq &  \left(\df{U(c)}{2}\right)^2 4 \left(c\sinh(4c)\right)^2 (1+ u(x)) \quad (\text{by} \  \eqref{eq:psiu}	) \notag \\
\leq &   \dfrac{\pi^2 e^{4c}  \left(c\sinh(4c)\right)^2}{\left(\sinh(2c)\right)^2(1+\cosh(4c))}  
\quad (\text{by Lemma} \  \ref{lem:upper_L}\text{ and }\eqref{eq:u0}).\notag  
\end{align}
We obtain \eqref{eq:F_bound} by the inequalities below. Using for the first display that, for $x\in\left[0,\pi/2\right]$, $\sin(x) \leq x$, \eqref{eq:psiu} and \eqref{eq:step11} for the second display, 
\eqref{eq:u0} and \eqref{eq:utilde1} for the third,  and Lemma \ref{lem:upper_L} for the fourth, we obtain,  for all  $x \in [0,1)$,
\begin{align}
\df{1-Y(x)^2}{p(x)} 
\leq &  \left(\df{\pi}{U(c)}\right)^2 \df{\left(\int_{x}^1p(\xi)^{-1/2} d\xi \right)^2}{p(x)}  \notag  \\
\leq &  \left(\dfrac{\pi}{U(c)}\right)^2 \dfrac{  2\left(1+ \widetilde{u}(x)\right)^2}{(4c\sinh(4c))^2} \dfrac{1}{1+u(x)}  \notag  \\
\leq &  \left(\dfrac{\pi}{U(c)}\right)^2 \dfrac{  (1+4c^2/3)^2  }{ 2c\left(\sinh(4c)\right)^2}\frac{\sinh(4c)}{\cosh(4c)-1}  \notag  \\
\leq & \dfrac{\pi^2e^{-4c}}{4c}\dfrac{ (1+4c^2/3)^2 \sinh(4c) }{ 
\cosh(4c)-1}.\notag
\end{align}
Classical relations between hyperbolic functions yield the final expressions for  \eqref{eq:F_1_bound} and \eqref{eq:F_bound}.
\end{proof}
The following constant appears in the susbsequent results
$$R(c)=\dfrac{2}{\pi^2} +\left(\dfrac{U(c)c}{\pi}\right)^2\left(\left(\cosh(4c)\left(1+ \dfrac{c}{3}\coth(2c)\right)-1\right) + 2c \sinh(4c)\right).$$
\begin{proposition}\label{t12i}
For all $c>0$ and $\lambda \in \C$, 
\eqref{eq:y}-\eqref{eq:gamma} maps a solution of $\left(H_{U(c)^2\lambda/\pi^2}\right)$  
in case (2) to a solution 
of $\left(H_{\lambda}\right)$ in case (1) and reciprocally the inverse 
maps a solution of ($H_{\lambda}$) in case (1) to a solution of $\left(H_{U(c)^2\lambda/\pi^2}\right)$ in case (2) and is a bijection of $\mathcal{D}$. 
Also, $q^c$ can be extended by continuity to $[-1,1]$ and, for all $ y\in[-1,1] $, 
\begin{equation}\label{eq:ineq}
\dfrac{1}{2} 
 - \left(\df{U(c)c}{\pi}\right)^2-R(c)
\leq q^c(y) \leq 
\dfrac{1}{2} 
 - \left(\df{U(c)c}{\pi}\right)^2.
\end{equation} 
\end{proposition}
\begin{proof}
Let $\Gamma$ and $\psi$ related via \eqref{eq:gamma}. 
By \eqref{eq:gamma}, we have 
$$F'(y)= \frac{F(y)}4\left(\frac{p'}{pY'}\left(Y^{-1}(y)\right)+\frac{2y}{1-y^2}\right)$$
$$(1-y^2)\Gamma'(y)= \frac{F(y)}4\left((1-y^2)\left(\frac{p'\psi}{pY'}+4\frac{\psi'}{Y'}\right)\left(Y^{-1}(y)\right)+2y\psi\left(Y^{-1}(y)\right)\right)$$
so differentiating a second time and injecting the above inequality, yields
\begin{align*}
\left((1-y^2)\Gamma'\right)'(y)=\frac{F(y)}{4}\Big(&\frac14
\left(\frac{p'}{pY'}\left(Y^{-1}(y)\right)+\frac{2y}{1-y^2}\right)
\left((1-y^2)\left(\frac{p'\psi}{pY'}+4\frac{\psi'}{Y'}\right)\left(Y^{-1}(y)\right)+2y\psi\left(Y^{-1}(y)\right)\right)\\
&-2y \left(\frac{p'\psi}{pY'}+4\frac{\psi'}{Y'}\right)\left(Y^{-1}(y)\right)+
(1-y^2)\left[\frac{1}{Y'}\left(\frac{p'\psi}{pY'}+4\frac{\psi'}{Y'}\right)'\right]\left(Y^{-1}(y)\right)\\
&+2\left(\psi\left(Y^{-1}(y)\right)+y\frac{\psi'}{Y'}\left(Y^{-1}(y)\right)\right)\Big).
\end{align*}
Dividing by $F(y)/4$ and using \eqref{eq:y}, $\Gamma$ is solution of $\left(H_{U(c)^2\lambda/\pi^2}\right)$  iff $\psi$ is solution on $(-1,1)$ of 
\begin{align*}
&\frac{1}{4p(x)}
\left(p'(x)+\frac{2YY'p}{1-Y^2}(x)\right)
\left(\frac{1-Y^2}{\left(Y'\right)^2p}(x)\left(p'\psi+4p\psi'\right)\left(x\right)+2\frac{Y}{Y'}(x)\psi\left(x\right)\right)\\
&-2\frac{Y}{Y'}(x) \left(\frac{p'\psi}{p}+4\psi'\right)(x)+
\frac{1-Y^2}{Y'}(x)\left(\frac{p'\psi}{pY'}+4\frac{\psi'}{Y'}\right)'(x)+2\left(\psi\left(x\right)+\frac{Y}{Y'}(x)\psi'\left(x\right)\right)\\
&=4\left(q^c\left(Y(x)\right)-\frac{U(c)^2\lambda}{\pi^2}\right)\psi\left(x\right).
\end{align*}
We now use, for all $x\in(-1,1)$,
\begin{align}
Y'(x) 
& = \dfrac{\pi }{U(c)p(x)^{1/2}} \cos(X(x)), \label{eq:YprimeE}
\end{align}
which yields the equality between $C^{\infty}$ functions:  
$(1-Y^2)/\left((Y')^2p\right)=\left(U(c)/\pi\right)^2$ and 
\begin{align*}
&
\left(1+2\frac{Y}{p'Y'}\left(\frac{\pi}{U(c)}\right)^2\right)
\left(\left(\frac{U(c)}{\pi}\right)^2\left(\frac{\left(p'\right)^2}{4p}\psi+p'\psi'\right)+\frac{Y}{2pp'Y'}\left(\left(p'\right)^2\psi\right)\right)\\
&-2\frac{Y}{p'Y'}\left(\frac{\left(p'\right)^2}{p}\psi+4p'\psi'\right)+
\left(\frac{U(c)}{\pi}\right)^2pY'\left(\frac{p'\psi}{pY'}+4\frac{\psi'}{Y'}\right)'+2\frac{Y}{p'Y'}p'\psi'\\
&=4\left(q^c\left(Y\right)-\frac12-\frac{U(c)^2\lambda}{\pi^2}\right)\psi.
\end{align*}
The term in factor of $\psi$ on the left-hand side of the above equality is
$$\left(\frac{U(c)}{\pi}\right)^2\left(1+2\frac{Y}{p'Y'}\left(\frac{\pi}{U(c)}\right)^2\right)^2\frac{\left(p'\right)^2}{4p}-2\frac{Y}{p'Y'}\frac{\left(p'\right)^2}{p}+\left(\frac{U(c)}{\pi}\right)^2\frac{pp''Y'-\left(p'\right)^2Y'-pp'Y''}{pY'}
$$
Using $-2pY''=p'Y'+2\left(\pi/U(c)\right)^2Y$ which is obtained from \eqref{eq:YprimeE}, this becomes
\begin{align*}
&\left(\frac{U(c)}{\pi}\right)^2\left(1+2\frac{Y}{p'Y'}\left(\frac{\pi}{U(c)}\right)^2\right)^2\frac{\left(p'\right)^2}{4p}-\frac{Y}{Y'}\frac{p'}{p}+\left(\frac{U(c)}{\pi}\right)^2\left(p''-\frac{\left(p'\right)^2}{2p}\right)\\
&=\left(\frac{Y}{Y'}\right)^2\left(\frac{\pi}{U(c)}\right)^2\frac{1}{p}+\left(\frac{U(c)}{\pi}\right)^2\left(p''-\frac{\left(p'\right)^2}{4p}\right)\\
&=\left(\tan\left(X(x)\right)\right)^2+\left(\frac{U(c)}{\pi}\right)^2\left(p''-\frac{\left(p'\right)^2}{4p}\right).
\end{align*}
hence
\begin{align*}
&4\left(\frac{U(c)}{\pi}\right)^2 (p\psi')'=4\left(q^c\left(Y\right)-\frac12-\frac14\left(\tan\left(X(x)\right)\right)^2+\frac14\left(\frac{U(c)}{\pi}\right)^2\left(\frac{\left(p'\right)^2}{4p}-p''\right)-\left(\frac{U(c)}{\pi}\right)^2\lambda\right)\psi
\end{align*}
and $\psi$ is solution of ($H_\lambda$) in case (1).\\ 
We now obtain upper and lower bounds on the even function $ q^c(Y(x))$, for  $ x\in[0,1]$, and start with the lower bound. 
To bound $\left(\tan(X)\right)^2$ in \eqref{eq:tan}, we 
use 
\begin{equation}\label{eq:tanb}
\left(\tan\left(\df{\pi}{U(c)}\int_{0}^{x}p(\xi)^{-1/2}d\xi\right)\right)^2 =  \left(\tan\left( \df{\pi}{U(c)} \int_{x}^{1}p(\xi)^{-1/2}d\xi \right) \right)^{-2},  
\end{equation} and 
(96) in \cite{yang2014sharp} in the first display and Lemma \ref{lem:lb} and the fact that $(a-b)^2 \geq a^2 - 2ab$ 
in the second display. 
We obtain 
\begin{align*}
\left(\tan\left(X(x)\right)\right)^2&  \geq \left(\df{U(c)}{\pi} \left(\int_{x}^{1} p(\xi)^{-1/2}d\xi\right)^{-1} - \dfrac{4}{\pi U(c)} \int_{x}^{1}p(\xi)^{-1/2}d\xi\right)^2 \\	
&  \geq \left(\df{U(c)c}{\pi}\right)^2 \left(\dfrac{\sinh(4c)}{c(1-x)} - \frac{8c\sinh(4c)\cosh(4c)}{3\left(\cosh(4c)-1\right)}\right) - \dfrac{8}{\pi^2}.
\end{align*}
To bound the second term in the bracket in \eqref{eq:tan} we proceed as follows. 
We have 
\begin{align} \dfrac{4c\left(\sinh(4cx)\right)^2}{p(x)} &= \dfrac{\sinh(4c)}{1-x}   \dfrac{1}{1+ u(x)} \dfrac{\left(\sinh(4cx)\right)^2}{\left(\sinh(4c)\right)^2}
\quad (\text{by}\ \eqref{eq:psiu}) 
\notag   \\
 &=  \dfrac{\sinh(4c)}{1-x} \left( 1 +\int_{x}^1 \dfrac{u'(t)dt}{(1+u(t))^{2}}\right)  \quad (\text{by}\ \eqref{eq:tool0})   \label{eq:step21} \\
 &\leq   \dfrac{\sinh(4c)}{1-x} \left( 1+ 8c^2(1-x) \right)  \quad (\text{by} \ \eqref{eq:tool})  \notag ,
\end{align}
hence
\begin{align*}
q^c(Y(x)) 
 \geq & \dfrac{1}{2} - \dfrac{2}{\pi^2} -\left(\dfrac{U(c)c}{\pi}\right)^2\left( \cosh(4c)\left(1+ \dfrac{2c\sinh(4c)}{3\left(\cosh(4c)-1\right)}\right) + 2c \sinh(4c) \right)\\
 \geq &\dfrac{1}{2} 
 - \left(\df{U(c)c}{\pi}\right)^2
 - R(c).
\end{align*}
Consider the upper bound on $ q^c$. For  $ x\in[0,1] $, by \eqref{eq:tanb} and $ 0<z\leq \tan(z) $ on $(0,\pi/2] $,  we have 
$$
q^c(Y(x)) \leq \dfrac{1}{2} +  \left(\df{U(c)c}{\pi}\right)^2 \left(\df{1}{ 4c^2\left( \int_{x}^{1}p(\xi)^{-1/2} d\xi\right)^2} - \df{\left(\sinh(4cx)\right)^2 }{p(x)} - \cosh(4cx) \right) .
$$
Using Lemma \ref{lem:lb}, \eqref{eq:step21}, and \eqref{eq:u}, 
we have
\begin{align}\label{eq:qq}
q^c(Y(x)) &\leq \dfrac{1}{2} - \left(\df{U(c)c}{\pi}\right)^2.
\end{align}
\end{proof}

The unbounded operator $\mathcal{L}$ on domain $\mathcal{D}$ in case (3) is self-adjoint. Indeed, it is shown page 571 of \cite{niessen1992singular} that $\mathcal{D}$ is the domain of the self-adjoint Friedrichs extension of the minimal operator corresponding to the differential operator on $L^2(-1,1)$ on the domain $\mathcal{D}_{\min}$ (the subset  of $\mathcal{D}_{\max}$ of functions with support in $(-1,1)$, see page 173 in \cite{zettl2005sturm}, we removed one condition on $\mathcal{D}_{\max} $ which is automatically satisfied).  By Proposition \ref{t12i}, 
the multiplication defined, for  $\psi \in \mathcal{D}_{\max}$, by $\psi\to q^c \psi$ is bounded and symmetric on $L^2(-1,1)$. 
Thus, by the Kato-Rellich theorem
(see, \emph{e.g.}, \cite{ReedSimon2}), the unbounded operator $\mathcal{L}$ on domain $\mathcal{D}$ in case (2) 
is self-adjoint. Denote by $\left((U(c)/\pi)^2\chi_{m}^{c}\right)_{m\in\N_0}$ 
the eigenvalues of the unbounded operator $\mathcal{L}$ on domain $\mathcal{D}$ in case (2) arranged in increasing order and repeated according to multiplicity.  They are real and, because the operator is bounded below, they are bounded below by the same constant. Moreover, Proposition \ref{t12i} yields that 
$\left(\chi_{m}^{c}\right)_{m\in\N_0}$ are the eigenvalues of the unbounded operator $\mathcal{L}$ on domain $\mathcal{D}$ in case (1). Proposition \ref{t12ii} gives a uniform behavior over $m$. It is in line with the asymptotic result on page 14 of \cite{Widom}. 
\begin{proposition} \label{t12ii}
We have, for all $m \in\N_0$ and $ c >0 $, 
\begin{equation*}
\left(\df{\pi}{U(c)}\right)^2 \left(m(m+1)+\frac12-R(c)\right)-c^2\leq \chi_{m}^{c} \leq 
\left(\df{\pi}{U(c)}\right)^2 \left(m(m+1)+\frac12\right)-c^2.
\end{equation*} 
\end{proposition}
\begin{proof}
The result is obtained by the min-max theorem and \eqref{eq:ineq}.
\end{proof}

\section{Uniform estimates on the singular functions $ g_{m}^{c} $}\label{sec:unif}
\begin{theorem}\label{prop:sup}
We have, for all $ m\in\N_0 $ and $c>0$, 
$$
\left\|g_m^c\right\|_{L^{\infty}([-1,1])}\le \dfrac{\pi}{e^{2c}} \left(1+\dfrac{4c^2}{3}\right)^{1/2} \left(\frac{\sinh(4c)}{4c}\right)^{1/4}\cosh(2c)^{1/2} \left(\df{2R(c)}{m+1/2}+\left(1+\sqrt{\frac23} \df{R(c)}{m+1/2} \right)\sqrt{m+\frac12}\right).
$$
\end{theorem}
\begin{proof}
The proof of this result relies on those of Section \ref{sec:diff}. Additional ingredients are common to those used in the proof of Proposition 5 in \cite{bonami2016uniform} but we repeat them for the sake of completeness. 
Using 
\eqref{eq:y} and 
\eqref{eq:gamma} with $\psi=g_m^c$,  and denoting by $\Gamma_m^c(\cdot)= F(\cdot)g_m^c\left(Y^{-1}(\cdot)\right)$ and 
$\widetilde{\Gamma}_m^c = \Gamma_m^c \sqrt{U(c)/\pi}$, which is real valued, in the first display, and \eqref{eq:YprimeE} and \eqref{eq:gamma} in the second display, we obtain
\begin{align}
\int_{-1}^{1} \left| \widetilde{\Gamma}_m^c(y) \right|^2 dy 
 &= \frac{U(c)}{\pi}\int_{-1}^{1} Y'(x) \left| F(Y(x)) \right|^2  \left| g_m^c(x) \right|^2 dx \notag\\
   &= \int_{-1}^{1} \dfrac{ \cos(X(x))}{\sqrt{1-\left(\sin(X(x))\right)^2}}   \left| g_m^c(x) \right|^2 dx= 1 .\notag 
\end{align}
Also, by Proposition \ref{t12i}, 
for all $y \in (-1,1)$,	
\begin{equation}\label{ednh}
\left((1-y^2)\left(\widetilde{\Gamma}_m^c\right)'\right)'(y) + m(m+1) \widetilde{\Gamma}_m^c(y)  = \left(m(m+1) - \left(\df{U(c)}{\pi}\right)^2\chi_{m}^{c}+ q^c(y)\right) \widetilde{\Gamma}_m^c(y).
\end{equation}
By the method of variation of constants, 
there exist $ A,B  \in  \R $ 
such that, for all $y\in(-1,1)$, 
\begin{equation}\label{eq:Gamma1}
\widetilde{\Gamma}_m^c(y) = A\overline{P}_m(y)  + BQ_m(y) + \df{1}{m+1/2} \int_{y}^1L_m(y,z)\sqrt{1-z^2}G_c(z)\widetilde{\Gamma}_m^c(z)dz,
\end{equation}
where
$ \overline{P}_m $ is the Legendre polynomial of degree $m$ and 
norm 1 in $L^2(-1,1)$, $ Q_m $ is the Legendre function of the second kind, 
 $ G_c(y)=  m(m+1) - (U(c)/\pi)^2\chi_{m}^{c} + q^c(y)$, and  
$  L_m(y,z)=\sqrt{1-z^2}\left(  \overline{P}_m(y)Q_m(z)  - \overline{P}_m(z)Q_m(y) \right)$.
By propositions \ref{t12i} and \ref{t12ii}, we have $ \left\| G_c \right\|_{L^{\infty}([-1,1])} \leq  R(c)$.  
Because $ \Gamma_m^c(1) $ is finite, $ \overline{P}_m $ is bounded but $ \lim_{y\to1}Q_m(y) =\infty$,  we know that $B=0$.
By the result after Lemma 9 in \cite{bonami2016uniform}, for all $0\le y\le z\le 1$, $|L_m(y,z)|\le 1$. Hence, by the Cauchy-Schwarz inequality, we have, for all $y \in (1,1)$,
\begin{align}
\left|\widetilde{\Gamma}_m^c(y) - A\overline{P}_m(y)\right| &\leq \dfrac{1}{m+1/2} \left(\int_{y}^1\left(L_m(y,z)\right)^2(1-z^2) dz\right)^{1/2} \left(\int_{y}^1G_c(z)^2\widetilde{\Gamma}_m^c(z)^2dz\right)^{1/2}, \notag\\
&\leq  \df{R(c)}{m+1/2}   (1-y)\label{eq:inter1}
\end{align}
so 
\begin{align*}
\int_{-1}^{1}\left|\widetilde{\Gamma}_m^c(y) - A\overline{P}_m(y)\right|^2 dy \leq  \df{2R(c)^2}{3(m+1/2)^2}\end{align*}
and, by the Cauchy-Schwarz inequality, 
\begin{align*}
\int_{-1}^{1}\left|\widetilde{\Gamma}_m^c(y) - A\overline{P}_m(y)\right|^2 dy  \geq& 1+A^2 -2|A|  \int_{-1}^{1}\left|\widetilde{\Gamma}_m^c(y)\right|^2 dy  \int_{-1}^{1} \left|\overline{P}_m(y)\right|^2 dy\\
\geq & (1-|A|)^2, 
\end{align*}
hence 
\begin{equation}\label{bA}
|A| \leq 1+\sqrt{\frac23} \df{R(c)}{m+1/2}.
\end{equation}
Also, by \eqref{eq:F_bound} and Lemma \ref{lem:upper_L}, we have 
\begin{align*}
\|1/F\|_{L^{\infty}([-1,1])}\sqrt{\dfrac{\pi}{U(c)}}  & \leq \pi e^{-2c} \left(1+\dfrac{4c^2}{3}\right)^{1/2} \left(\frac{\sinh(4c)}{4c}\right)^{1/4}\cosh(2c)^{1/2},
\end{align*}
and we obtain the result by \eqref{eq:inter1}, \eqref{bA}, and $\left\| \overline{P}_m \right\|_{L^{\infty}([-1,1])} \leq \sqrt{m+1/2}$.
\end{proof}
\begin{corollary}
For all $ m\in\N_0 $ and $c>0$, 
\begin{equation}\label{bunif} 
\left\| g_{m}^{c}\right\|_{L^{\infty}([-1,1])} \leq H(c) \sqrt{m+\df{1}{2}}, 
\end{equation}
where 
$$H(c)= \pi\sqrt{1+\dfrac{4c^2}{3}}\left(1+ 2\sqrt{2}\left(2+ \df{1}{\sqrt{3}}\right) \left(\dfrac{2}{\pi^2} +\frac{8}{3}\left(1+2c\right)\left(c^2 + \dfrac{9c}{8} + \frac{1}{2}\right)\right)  \right).$$
\end{corollary}
\begin{proof}
By the above results, \eqref{bunif} holds with
$$
 \pi e^{-2c} \left(1+\dfrac{4c^2}{3}\right)^{1/2} \left(\frac{\sinh(4c)}{4c}\right)^{1/4}\cosh(2c)^{1/2}\left(1+ 2\sqrt{2}R(c) \left(2+ \df{1}{\sqrt{3}}\right)  \right)
$$
in place of $H(c)$ and 
$$R(c)<\dfrac{2}{\pi^2} +2\left(\frac{c e^{2c}}{\sinh(4c)}\right)^2\left(\left(\cosh(4c)\left(1+ \dfrac{c}{3}\coth(2c)\right)-1\right) + 2c \sinh(4c)\right).$$
hence, using that $e^c\ge1+c$ which implies  $c\coth(c)\le c+2$, 
$$R(c)<\dfrac{2}{\pi^2} +\frac{8c e^{4c}}{3\sinh(4c)}\left(c^2 + \dfrac{9c}{8} + \frac{1}{2}\right)<\dfrac{2}{\pi^2} +\frac{8}{3}\left(1+2c\right)\left(c^2 + \dfrac{9c}{8} + \frac{1}{2}\right).$$
We obtain the result, using 
\begin{align*}
e^{-2c} \left(\frac{\sinh(4c)}{4c}\right)^{1/4}\cosh(2c)^{1/2}&=e^{-2c} \left(\frac{\sinh(2c)}{2c}\right)^{1/4}\cosh(2c)^{3/4}\\&=\left(\frac{1-e^{-4c}}{4c}\right)^{1/4}\left(
\frac{1+e^{-4c}}{2}\right)^{3/4}\le 1.
\end{align*}
\end{proof}
As a result we have, for a constant $C_0$, 
\begin{equation}\label{bunif0} 
\left\| g_{m}^{c}\right\|_{L^{\infty}([-1,1])} \leq C_0 \left(c \vee 1 \right)^4 \sqrt{m+\dfrac{1}{2}}. 
\end{equation}
Such a square-root bound also holds for Legendre polynomials and PSWF. It allows for compressed sensing recovery estimates (see Sections 3.2 and 4.1 in \cite{Gosse13}). It is possible that the factor $\left(c \vee 1 \right)^4$ could be improved. For the PSWF, \cite{bonami2016uniform} obtain a factor $c^2$. Showing sharpness in $c$ in either case would require a lower bound. This bound is used in \cite{estimation}, even for diverging $c$, to obtain an adaptive estimator. It could be used to justify the data-driven selection rule in Section \ref{sec:sim}.   

\section{Numerical method to obtain the SVD of $\mathcal{F}_{b,c}$}\label{sec:nmerical}
In recent years, efficient numerical methods to obtain the SVD of the truncated Fourier transform acting on the space of bandlimited functions have been developed. This section presents how to deal similarly with nonbandlimited functions.

The strategy we implement in Section \ref{sec:sim} is to first compute a numerical approximation of the right singular functions (the PSWF).
We use that the first coefficients of the decomposition of the PSWF on the Legendre polynomials can be obtained by solving for the eigenvectors of two tridiagonal symmetric Toeplitz matrices (for even and odd values of $m$, see Section 2.6 in \cite{Osipov}).  
We can then compute their image by $\mathcal{F}_c^{W_{[-1,1]}*}$ (see \eqref{eqdefFpSWF} for the definition of $\mathcal{F}_{c}^{W_{[-1,1]}}$) because, by \cite{estimation}, $\mathcal{F}_c^{W_{[-1,1]}*}=\mathcal{R}\left[\indic\{[-1,1]\}\mathcal{F}_{c}^{W_{[-1,1]}}\mathcal{E}\right]$ applied to the 
Legendre polynomials 
has a closed form involving the Bessel functions of the first kind (see (18.17.19) in \cite{olver2010nist}). 

For nonbandlimited functions, we propose to rely on the differential operator $ \mathcal{L}$ in case (1) at the beginning of Section \ref{sec:diff}. We have used that because $\mathcal{Q}_{c}$ commutes with $ \mathcal{L}$,  
$\left(g_m^{c}\right)_{m\in\N_0}$ are the eigenfunctions of $\mathcal{L}$. 
To obtain a numerical approximation of these functions, 
we use $ \mathcal{L}$, whose eigenvalues 
are of the order of $m^2$ (see Proposition \ref{t12ii}), rather than $\mathcal{Q}_{c}$, whose eigenvalues decay to zero  exponentially. This is achieved by solving numerically for the eigenfunctions of a singular Sturm-Liouville operator.  
We approximate the values of the eigenfunctions on a grid on $[-1,1]$ using the MATLAB package MATSLISE 2 (it implements constant perturbation methods for limit point nonoscillatory singular problems, see \cite{ledoux2007study} chapters 6 and 7 for the method and an analysis of the numerical approximation error). By Proposition A.1 in \cite{estimation}, we have $\varphi_m^{b,c}(\cdot)=\varphi_m^{1,c/b}(b \cdot)\sqrt{b}$ for all $m\in\N_0$. 
Finally, we use $\mathcal{F}_{1,c/b}^*\left[g_m^{c/b}\right]=\sigma_m^{1,c/b}\varphi_m^{1,c/b}$ and that $\varphi_m^{1,c/b}$ has norm 1 to obtain the remaining of the SVD.  $\mathcal{F}_{1,c/b}^*\left[g_m^{c/b}\right]$ is computed using the fast Fourier transform. 

\section{Illustration: application to analytic continuation}\label{sec:sim}
We solve for $f$ in \eqref{eq:f} in Case (a) $ f= 0.5/\cosh(2\cdot)$, which is not bandlimited, and Case (b)  $f = \text{sinc}(2\cdot)/6$ which is bandlimited, when $c=0.5$, $x_0=0$, and $\xi = \cos(50 \cdot)$.   
We use the 
approximation $f_{\delta}^{N}$ described in Section \ref{sec:analy} 
with $b=1$ for Case (a), $b=1/6.5$ for Case (b). By analogy with the statistical problem where $\delta\xi$ is random rather than bounded,  we use the terminology estimator. 

In Case (a), we only consider a reconstruction using the functions $g_m^{c/b}$ because $f$ is not bandlimited. Proving that PSWF for large $b$ can perform well in this case would require a careful analysis of how well such expansions can approximate nonbandlimited functions. This is out of scope of this paper.  
In  Case (b), we compare  $f_{\delta}^{N}$  to a similar estimator based on \eqref{eq:ft1} but with the PSWF instead of $g_m^c$. This approach can only be used to perform analytic continuation of bandlimited functions when the researcher knows an interval which contains the bandlimits. 
In contrast, even for bandlimited functions, using the estimator based on $g^{c/b}_m$ allows to perform analytic continuation without the knowledge of an interval containing the support of the Fourier transform of the function. 
Figure \ref{fig:Sp2} shows that $f_{\delta}^{N}$ performs almost as well as the ``oracle" method which uses the PSWF and precise knowledge of the bandlimits. 

Let us now consider a practical feasible selection method for the parameter $N=\widehat{N}$ which does not depend on the unknowns but which uses the SVD. 
We use a type of Goldenshluger-Lepski method similar to the one in \cite{estimation} which delivers an optimal minimax estimator for the generalization of the tomography problem: \begin{align*}
 &\widehat{N} \in \underset{N' \in \{0,\dots, N_{\max}\}}{\text{argmin}} B(N) + \Sigma(N),    \\
 & B(N) = \sup_{N \leq N' \leq N_{\max}} \left(\left\|  F_{\delta}^{N'\vee N} -   F_{\delta}^{ N}\right\|^2_{L^2(\cosh(b\cdot))} + \Sigma(N') \right)_{+},\  \Sigma(N) =  \dfrac{  2\pi  c \delta^2 e^{2\beta(c/b) N}}{1 - e^{-2\beta(c/b)}} ,
\end{align*}
and $N_{\max} =\lfloor \log(1/\delta)\rfloor$. 
Performing analytic continuation using \eqref{eq:ft1} requires the approximation of the scalar products on $[-1,1]$ of the observed function $f_{\delta}$ with $g_m^{c/b}$. We use the package MATSLISE 2 to compute the value of the functions $\left(g_m^{c/b}\right)_{m=0}^{N_{\max}}$ at the $n$ first Gauss-Legendre quadrature nodes. Results are presented in  figures \ref{fig:Sp4} and \ref{fig:Sp3}, where we use a $2^{12}$ resolution in the Fast Fourier transform, $n=15000$, and precision of $10^{-10}$ for the computation of the eigenvalues in MATSLISE 2, which also controls the precision of the computation of the eigenfunctions in the function computeEigenfunction of MATSLISE 2 despite that this is not explicitly computed (see sections 7.2.3 and 5.2 in \cite{ledoux2007study} for examples). The numerical results show that the data-driven choice of the degree of regularization $\widehat{N}$ works well in practice. A classical criticism is that using a data-driven truncated SVD can be costly. 
However, $N_{\max}$ only depends on the noise level and grows logarithmically.  In practice, we obtain $\widehat{N}\le N_{\max}$ by computing in advance a relatively small number of singular values and functions. The approach of Section \ref{sec:nmerical} is fast to implement, so the computational price of computing the SVD is a small price to pay for having a feasible data-driven smoothing method which we expect to be optimal.

For the sake of conciseness, this paper does not study the effect of the various discretizations which can be carried out with arbitrary precision. Rather, we used in the numerical illustration conservative choices for those. This paper also does not consider the statistical problem, prove minimax lower bounds for it, and the adaptivity of the data-driven rule giving $N=\widehat{N}$. This is the object of future work. The interested reader can refer to \cite{estimation} for the full statistical analysis for estimation of the density of random coefficients in the linear random coefficients model.  

\begin{figure}[H]
\centering
\includegraphics[width=0.85\linewidth, height=0.25\textheight]{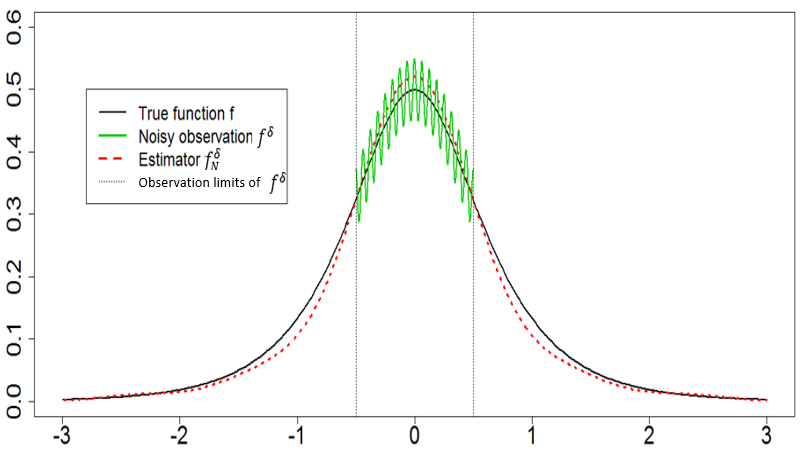}
\caption{
Case (a) with noise ($\delta=0.05$), where $F_{\delta}^N$ in \eqref{eq:ft1} uses $g^{c/b}_m$.}
\label{fig:Sp4}
\end{figure}

\begin{figure}[H]
\centering
\includegraphics[width=0.85\linewidth, height=0.25\textheight]{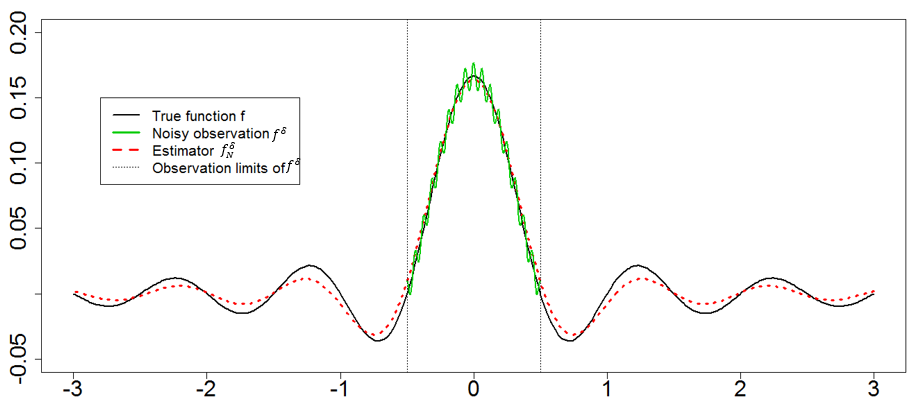}
\caption{
Case (b) with noise ($\delta=0.01$), where  $F_{\delta}^N$ in \eqref{eq:ft1} uses $g^{c/b}_m$.}
\label{fig:Sp3}
\end{figure}
\begin{figure}[H]
\centering
\includegraphics[width=0.85\linewidth, height=0.25\textheight]{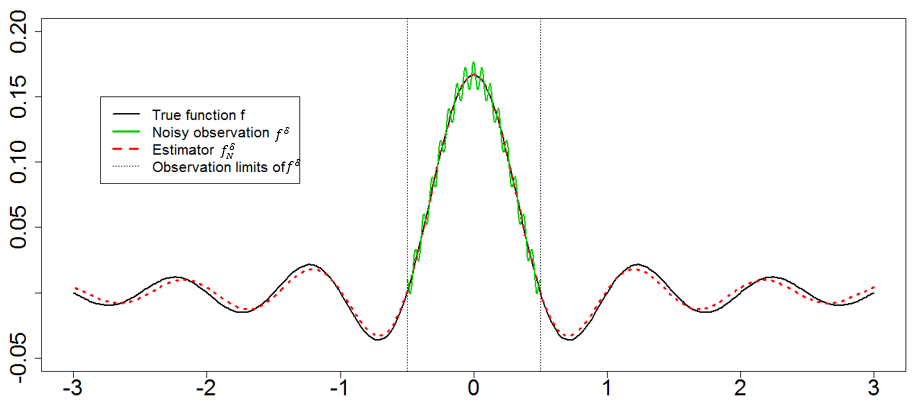}
\caption{
Case (b) with noise ($\delta=0.01$), where  $F_{\delta}^N$ in \eqref{eq:ft1} uses the PSWF.}
\label{fig:Sp2}
\end{figure}

\bibliographystyle{abbrv}
\bibliography{bib_02_5}


\end{document}